\documentclass[11pt]{amsart}
\usepackage{amsmath}
\usepackage{amssymb}
\usepackage{amsfonts}
\usepackage{amsthm}
\usepackage{enumerate}
\usepackage[mathscr]{eucal}
\usepackage{verbatim}
\usepackage{amsthm}
\usepackage{amscd}

\usepackage{enumitem}

\makeatletter
\def\step{
   \@ifnextchar[ \@step{\@noitemargtrue\@step[\@itemlabel]}}
\def\@step[#1]{\item[#1]\mbox{}\\\hspace*{\dimexpr-\labelwidth-\labelsep}}
\makeatother

\newtheorem{theorem}{Theorem}[section]

\newtheorem{lemma}[theorem]{Lemma}

\theoremstyle{definition}

\newtheorem{remark}{Remark}

\numberwithin{equation}{section}

\begin{document}

\address{School of Mathematics \\
           Korea Institute for Advanced Study, Seoul\\
           Republic of Korea}
   \email{qkrqowns@kias.re.kr}
\author{Bae Jun Park}

\title[Fourier Multipliers]{ Fourier multiplier theorems for Triebel-Lizorkin spaces}
\keywords{Fourier Multipliers,  Triebel-Lizorkin spaces,  H\"ormander-Mikhlin Multipliers}

\begin{abstract} 
In this paper we study sharp generalizations of $\dot{F}_p^{0,q}$ multiplier theorem of Mikhlin-H\"ormander type. 
The class of multipliers that we consider involves Herz spaces $K_u^{s,t}$. Plancherel's theorem proves $\widehat{L_s^2}=K_2^{s,2}$ and we study the optimal triple $(u,t,s)$ for which $\sup_{k\in\mathbb{Z}}{\big\Vert \big( m(2^k\cdot)\varphi\big)^{\vee}\big\Vert_{K_u^{s,t}}}<\infty$ implies $\dot{F}_p^{0,q}$ boundedness of multiplier operator $T_m$ where  $\varphi$ is a cutoff function. Our result also covers the $BMO$-type space $\dot{F}_{\infty}^{0,q}$.
\end{abstract}

\maketitle
\section{\textbf{Introduction}}

In this paper we give some sharp estimates for multipliers of  Mikhlin-H\"ormander type  in Triebel-Lizorkin spaces $\dot{F}_p^{0,q}$. 
Let $S(\mathbb{R}^d)$ denote the Schwartz space and $S'(\mathbb{R}^d)$ the space of tempered distributions. 
For the Fourier transform of $f$ we use the definition 
$\widehat{f}(\xi):=\int_{\mathbb{R}^d}{f(x) e^{-2\pi i\langle x,\xi \rangle}}dx $ and denote by $f^{\vee}$ the inverse Fourier transform of $f$.

For $m\in L^{\infty}$  the multiplier operator $T_m$ is defined as $T_mf(x):=\big( m\widehat{f}\big)^{\vee}(x)$.
The classical Mikhlin multiplier theorem \cite{Mik} states that if a function $m$, defined on $\mathbb{R}^d$, satisfies
\begin{eqnarray*}
\big|  \partial_{\xi}^{\beta}m(\xi)  \big|\lesssim_{\beta}|\xi|^{-|\beta|}
\end{eqnarray*} for all multi-indices $\beta$ with $|\beta|\leq \big[d/2\big]+1$, then the operator $T_m$ is bounded on $L^p$ for $1<p<\infty$.
In \cite{Ho} H\"ormander extends  Mikhlin's theorem to functions $m$ with the weaker condition
\begin{eqnarray}\label{hocondition}
\sup_{k\in\mathbb{Z}}{\big\Vert m(2^k\cdot)\varphi\big\Vert_{L^2_s}}<\infty
\end{eqnarray} for $s>d/2$ where $L^2_s$ stands for the standard fractional Sobolev space, $\varphi$ is a cutoff function such that $0\leq \varphi\leq 1$, $\varphi=1$ on $1/2\leq |\xi|\leq 2$, and $Supp(\varphi)\subset \{1/4\leq |\xi|\leq 4\}$.
When $0<p\leq 1$ Calder\'on and Torchinsky \cite{Ca_To} proved that if (\ref{hocondition}) holds for $s>d/p-d/2$, then $m$ is a Fourier multiplier of Hardy space $H^p$. A different proof was given by Taibleson and Weiss \cite{Ta_We}. Moreover, Baernstein and Sawyer \cite{Ba_Sa} sharpened these results by using Herz space conditions for  $\big(m(2^k\cdot)\varphi\big)^{\vee}$.

We recall the definition of Herz spaces  \cite{He}, following the terminology in \cite{Ba_Sa}.
Let $B_0:=\{x:|x|\leq 2\}$ and 
for $k\in \mathbb{Z}^+$ let $B_k:=\{x\in\mathbb{R}^d: 2^k< |x|\leq 2^{k+1}\}$.  Suppose $0< u\leq \infty$, $t>0$, and $s\in\mathbb{R}$. Then the Herz space  $K_u^{s,t}$  is the collection of $f\in L_{loc}^{u}$ such that 
\begin{eqnarray*}
\Vert f\Vert_{K_u^{s,t}}:=\Big( \sum_{k=0}^{\infty}{2^{skt}\Vert f\Vert_{L^u(B_k)}^t }\Big)^{1/t}<\infty
\end{eqnarray*}
with the usual modification if $t=\infty$.
Then elementary considerations show that $K_u^{s_1,t_1}\hookrightarrow K_u^{s_2,t_2}$ for $s_2<s_1$, $K_u^{s_1,t_1}\hookrightarrow K_u^{s_2,t_2}$ for $t_1<t_2$, and $K_{u_1}^{s,t_1}\hookrightarrow K_{u_2}^{u,t_2}$ for $u_2\leq u_1$ with $s_2=s_1-(d/u_2-d/u_1)$. 
The condition (\ref{hocondition}) is equivalent to 
\begin{eqnarray}\label{kcondition}
\sup_{k\in\mathbb{Z}}\big\Vert \big(m(2^k\cdot)\varphi \big)^{\vee}\big\Vert_{K_2^{s,2}}<\infty,
\end{eqnarray}
 and 
Baernstein and Sawyer \cite{Ba_Sa} improved the $H^p$ boundedness by replacing $K_2^{s,2}$ in (\ref{kcondition}) by $K_1^{d/p-d,p}$, which is an endpoint result,  when $0<p<1 $. 
Note that
$$K_2^{s,2}\hookrightarrow K_2^{d/p-d/2,p}\hookrightarrow K_u^{d/p-d/u,p}\hookrightarrow K_p^{0,p}$$ for $s>d/p-d/2$ and $p\leq u\leq 2$.\\

The main purpose of this paper is to provide some improvements of the results by Baernstein and Sawyer.
Let \begin{eqnarray*}
\mathcal{K}_{u}^{s,t}[m]:=\sup_{k\in\mathbb{Z}}\big\Vert \big(m(2^k\cdot)\varphi\big)^{\vee}\big\Vert_{K_u^{s,t}}<\infty.
\end{eqnarray*}
 We will study the optimal condition of triples $(u,t,s)$ for which   $\mathcal{K}_{u}^{s,t}[m]<\infty$ implies $\dot{F}_p^{0,q}$ boundedness of $T_{m}$ for $0<p,q\leq \infty$. 
Note that  $\dot{F}_p^{0,2}=H^p$, $\dot{F}_p^{\alpha,2}=H_{\alpha}^{p}$ for $0<p<\infty$, $\alpha\in\mathbb{R}$ and  $\dot{F}_{\infty}^{0,2}=BMO$ where $H_{\alpha}^{p}$ stands for the Hardy-Sobolev space.

The example by Baernstein and Sawyer \cite{Ba_Sa} shows that the $H^1$ boundedness of $T_m$ fails with the condition $\mathcal{K}_1^{0,1}[m]<\infty$. Instead, they applied weighted Herz spaces to establish an endpoint estimate for $H^1$ multipliers.
Seeger \cite{Se2} extended this result to $\dot{F}_p^{0,q}$, $p,q\geq 1$. 
Another main part of this paper is to improve these results using more generalized weighted Herz spaces.\\

This paper is organized as follows.
 We describe our multiplier results in Section \ref{mainsection}, sharpness of them in Section \ref{negative}, and endpoint multiplier results with weighted Herz space conditions in Section \ref{weight}. The proof of the results will be given in Section \ref{mmultiplier}, \ref{example}, and \ref{weightproof}.

\section{\textbf{Multiplier theorems}}\label{mainsection}

Let us start by recalling the definition of Besov sapces and Triebel-Lizorkin spaces.
 Let $\phi$ be a smooth function so that $\widehat{\phi}$ is supported in $\{\xi:2^{-1}\leq |\xi|\leq 2\}$ and $\sum_{k\in\mathbb{Z}}{\widehat{\phi_k}(\xi)}=1$ for $\xi\not=0$ where $\phi_k:=2^{kd}\phi(2^k\cdot)$. 
 For each $k\in\mathbb{Z}$ we define convolution operators $\Pi_k$ by $\Pi_kf:=\phi_k\ast f$.
For $0<p,q\leq \infty$ and $\alpha\in \mathbb{R}$. The (homogeneous) Besov spaces $\dot{B}_p^{\alpha,q}$ and Triebel-Lizorkin spaces $\dot{F}_p^{\alpha,q}$  are defined as  subspaces of $S'/\mathcal{P}$ (tempered distributions modulo polynomials) with  (quasi-)norms   $$\Vert f\Vert_{\dot{B}_p^{\alpha,q}}:=\big\Vert \{2^{\alpha k}\Pi_kf\}_{k\in\mathbb{Z}}\big\Vert_{l^q(L^p)},$$ 
 $$\Vert f\Vert_{\dot{F}_p^{\alpha,q}}=\big\Vert \{2^{\alpha k}\Pi_kf\}_{k\in\mathbb{Z}}\big\Vert_{L^p(l^q)}, ~ p<\infty ~\text{or}~p=q=\infty, $$
 respectively.
 When $p=\infty$ and $q<\infty$ we apply
$$\Vert f\Vert_{\dot{F}_{\infty}^{\alpha,q}}:=\sup_{P\in\mathcal{D}}\Big(\frac{1}{|P|}\int_P{\sum_{k=-\log_2{l(P)}}^{\infty}{2^{\alpha kq}|\Pi_kf(x)|^q}}dx \Big)^{1/q}$$ where $\mathcal{D}$ stands for the set of all dyadic cubes in $\mathbb{R}^d$ and $l(P)$ means the side length of $P\in\mathcal{D}$. According to those norms,  the spaces are quasi-Banach spaces (Banach spaces if $p\geq 1, q\geq 1$).\\

Since the set of all Fourier multipliers for $\dot{F}_p^{\alpha,q}$ is independent of $\alpha$ (similarly for $\dot{B}_p^{\alpha,q}$) we shall deal with only the case $\alpha=0$ in this paper.\\

First let us  state multiplier theorems for Besov spaces $\dot{B}_p^{0,q}$. 
\begin{theorem}\label{multiplierbesov}
Assume $0<p,q\leq \infty$ and let $r:=\min{(1,p)}$.
\begin{enumerate}
\item For $0<u\leq r$, 
\begin{eqnarray*}
\big\Vert T_mf\big\Vert_{\dot{B}_p^{0,q}}\lesssim \mathcal{K}_u^{0,r}[m]\Vert f\Vert_{\dot{B}_p^{0,q}}.
\end{eqnarray*}
\item For $r<u\leq \infty$
\begin{eqnarray*}
\big\Vert T_mf\big\Vert_{\dot{B}_p^{0,q}}\lesssim \mathcal{K}_{u}^{d/r-d/u,r}[m]\Vert f\Vert_{\dot{B}_p^{0,q}}.
\end{eqnarray*}
\end{enumerate}
\end{theorem}

\begin{remark}
By embedding $K_u^{d/r-d/u,r}\hookrightarrow K_r^{0,r}$ for $r<u$, (2) is immediate from (1). Moreover, these results are sharp in the sense that $\mathcal{K}_u^{0,r}[m]$ in (1) and $\mathcal{K}_u^{d/r-d/u,r}[m]$ in (2) cannot be replaced by
$\mathcal{K}_u^{0,r+\epsilon}[m]$ and $\mathcal{K}_u^{d/r-d/u,r+\epsilon}[m]$, respectively. This optimality of $t=r$ in $\mathcal{K}_u^{s,t}$ also implies that of $s=0$ in (1) and $s=d/r-d/u$ in (2).
Some related examples will be given in Section \ref{negative}.
\end{remark}

 $\dot{F}_p^{0,q}$ multiplier theorems for the case $p=q$ follow from Theorem \ref{multiplierbesov} because $\dot{F}_p^{0,p}=\dot{B}_p^{0,p}$. Therefore we now consider the case $p\not= q$.

\begin{theorem}\label{multipliertheorem0}
Assume  $0<p,q\leq \infty$, $p\not= q$, and  $r:=\min{(p,q)}<1$. Let $0<u\leq r$ and $t>0$.
Then for $s>0$ $$\mathcal{K}_{u}^{s,t}[m]<\infty$$
implies the $\dot{F}_p^{0,q}$ boundedness of $T_m$.
Moreover, in the case
\begin{eqnarray*}
\big\Vert T_mf\big\Vert_{\dot{F}_{p}^{0,q}}\lesssim \mathcal{K}_{u}^{s,t}[m]\Vert f\Vert_{\dot{F}_{p}^{0,q}}.
\end{eqnarray*}
\end{theorem}

\begin{theorem}\label{multipliertheorem1}
Assume  $0<p,q\leq \infty$, $p\not= q$, and $r:=\min{(p,q)}<1$. Let $r<u\leq \infty$, $t>0$, and $s\in\mathbb{R}$.
Suppose that  $m\in L^{\infty}$ satisfies
$$\mathcal{K}_{u}^{s,t}[m]<\infty.$$
Then $T_m$ is bounded on $\dot{F}_p^{0,q}$ if one of the following conditions holds;
\begin{enumerate}
\item $s=d/r-d/u$ and $t\leq r$, 
\item $s>d/r-d/u$.
\end{enumerate}
Moreover, in both cases
\begin{eqnarray*}
\big\Vert T_mf\big\Vert_{\dot{F}_{p}^{0,q}}\lesssim \mathcal{K}_{u}^{s,t}[m]\Vert f\Vert_{\dot{F}_{p}^{0,q}}.
\end{eqnarray*}
\end{theorem}

\begin{theorem}\label{multipliertheorem2}
Assume $1\leq p,q\leq \infty$ and $p\not= q$. Let $0<u\leq 1$ and $t>0$. 
Then for $s>0$  $$\mathcal{K}_{u}^{s,t}[m]<\infty$$ implies the
$\dot{F}_p^{0,q}$ boundedness of $T_m$.
Moreover, in the case \begin{eqnarray*}
\big\Vert T_mf\big\Vert_{\dot{F}_{p}^{0,q}}\lesssim \mathcal{K}_{u}^{s,t}[m]\Vert f\Vert_{\dot{F}_{p}^{0,q}}.
\end{eqnarray*} 
\end{theorem}

\begin{theorem}\label{multipliertheorem3}
Assume $1\leq p,q\leq \infty$ and  $p\not= q$. Let $1<u\leq \infty$, $t>0$, and $s\in\mathbb{R}$. Suppose that  $m\in L^{\infty}$ satisfies $$\mathcal{K}_{u}^{s,t}[m]<\infty.$$
Then $T_m$ is bounded on $\dot{F}_p^{0,q}$ if one of the following conditions holds;
\begin{enumerate}
\item  $s=d-d/u$, $t\leq 1$, $1<p<\infty$ and $\big|1/p-1/q\big|<1-1/u$, 
\item $s>d-d/u$.
\end{enumerate}
Moreover, in both cases \begin{eqnarray*}
\big\Vert T_mf\big\Vert_{\dot{F}_{p}^{0,q}}\lesssim \mathcal{K}_{u}^{s,t}[m]\Vert f\Vert_{\dot{F}_{p}^{0,q}}.
\end{eqnarray*} 
\end{theorem}

\begin{remark}
The second assertions in Theorem \ref{multipliertheorem1} and \ref{multipliertheorem3} follow from Theorem \ref{multipliertheorem0} and \ref{multipliertheorem2}, respectively, and thus we will only prove the endpoint cases $s=d/\min{(1,p,q)}-d/u$.
In addition, the sharpness of the above theorems will be discussed in Section \ref{negative}.

\end{remark}

\subsection{Inhomogeneous versions}
 We recall the definition of inhomogeneous Besov spaces $B_p^{\alpha,q}$ and Triebel-Lizorkin spaces $F_p^{\alpha,q}$.
Let $\{\phi_k\}$ be a dyadic resolution of unity as before, and let $\widehat{\Phi_0}:=1-\sum_{k=1}^{\infty}{\widehat{\phi_k}}$.
Then we define a convolution operator $\Lambda_0$ by $\Lambda_0f:=\Phi_0\ast f$ and let $\Lambda_k=\Pi_k$ for $k\geq 1$.   
For $\alpha\in \mathbb{R}$ and $0<p,q\leq \infty$ $B_p^{\alpha,q}$ and $F_p^{\alpha,q}$ are the collection of all $f\in S'$ such that
\begin{eqnarray*}
\Vert f\Vert_{B_p^{\alpha,q}}:=\big\Vert \big\{ 2^{\alpha k}\Lambda_kf\big\}\big\Vert_{l^q(L^p)}<\infty
\end{eqnarray*} 
\begin{eqnarray*}
\Vert f\Vert_{F_p^{\alpha,q}}:=\big\Vert \big\{  2^{\alpha k}\Lambda_kf\big\}\big\Vert_{L^p(l^q)}<\infty, \quad p<\infty
\end{eqnarray*} respectively.
When $p=q=\infty$ we employ $F_{\infty}^{\alpha,q}=B_{\infty}^{\alpha,\infty}$ and when $p=\infty$ and $q<\infty$ 
\begin{eqnarray*}
\Vert f \Vert_{F_{\infty}^{\alpha,q}}:=\Vert \Lambda_0f\Vert_{L^{\infty}}+\sup_{P\in\mathcal{D}, l(P)<1}{\Big(\frac{1}{|P|}\int_P{\sum_{k=-\log_2{l(P)}}^{\infty}{2^{\alpha kq}|\Lambda_kf(x)|^q}}dx \Big)^{1/q}}
\end{eqnarray*}
where the supremum is taken over all dyadic cubes whose side length is less than $1$.\\

Let $\psi$ be a Schwartz function so that $0\leq \psi\leq 1$, $Supp(\psi)\subset \{\xi:|\xi|\leq 2\}$, and $\psi(\xi)=1$ for $|\xi|\leq 1$.
Now we define 
\begin{eqnarray*}
\mathfrak{K}_u^{s,q}[m]:=\big\Vert (m\psi)^{\vee}\big\Vert_{K_u^{s,q}}+\sup_{k\geq 1}{\big\Vert \big(m(2^k\cdot)\varphi\big)^{\vee}\big\Vert_{K_u^{s,q}}},
\end{eqnarray*} which is an inhomogeneous modification of $\mathcal{K}_u^{s,q}[m]$.\\

In \cite{Tr3} and \cite[p74]{Tr} Triebel proved that
for $0<p,q<\infty$
if $m\in L^{\infty}$ satisfies \begin{eqnarray*}
\mathfrak{K}_2^{s,2}[m]<\infty, \quad s>d/\min{(1,p,q)}-d/2
\end{eqnarray*} then
\begin{eqnarray*}
\Vert T_m\Vert_{F_p^{0,q}}\lesssim \mathfrak{K}_2^{s,2}[m]\Vert f\Vert_{F_p^{0,q}}.
\end{eqnarray*}
It was first proved that for $1<p,q<\infty$ if $\mathfrak{K}_2^{s,2}[m]<\infty$ for $s>0$ then $T_m$ is bounded on $F_p^{0,q}$ by using the classical H\"ormander-Mikhlin multiplier theorem. 
Moreover, for $0<p,q<\infty$ it is easy to obtain the $F_p^{0,q}$ boundedness of $T_m$ under the assumption $\mathfrak{K}_2^{s,2}[m]<\infty$ with $s>d/2+d/\min{(p,q)}$. 
Then Triebel \cite{Tr3}, \cite[p74]{Tr} applied a complex interpolation method to derive $s>d/\min{(1,p,q)}-d/2$.\\

All of our results in homogeneous spaces easily adapt to the inhomogeneous cases by replacing $\mathcal{K}_u^{s,q}[m]$ by $\mathfrak{K}_u^{s,q}[m]$, and these results definitely improve and generalize Triebel's results.
We note that the complex interpolation method in \cite{Tr3} cannot be applied to the case $p=\infty$. Basically, 
the method in this paper is totally different from that used in \cite{Tr3} and our method does not apply previous multiplier theorems. 
Moreover the sharpness of $s>d/\min{(1,p,q)}-d/2$ was  left open in \cite{Tr3}, especially for $q<1$ and $q<p$, and our examples in Section \ref{example}  address this issue.

\section{\textbf{Some negative results}}\label{negative}

In this section we study the sharpness of the conditions on $s$ and $t$ in Theorem \ref{multiplierbesov}$-$\ref{multipliertheorem3}.

\begin{theorem}\label{sharpbesov}
Let $0<p,q\leq \infty$ and let  $r:=\min{(1,p)}$.
\begin{enumerate}
\item For $0<u\leq r$ and $t>r$ there exists $m\in L^{\infty}$ so that $\mathcal{K}_u^{0,t}[m]<\infty$, but $T_m$ is not bounded on $\dot{B}_p^{0,q}$.

\item For $ r<u\leq \infty$ and $t>r$ there exists $m\in L^{\infty}$ so that $\mathcal{K}_u^{d/r-d/u,t}[m]<\infty$, but $T_m$ is not bounded on $\dot{B}_p^{0,q}$.
\end{enumerate}
\end{theorem}

\begin{theorem}\label{sharpthm0}
Let $0<p,q\leq \infty$, and $p\not= q$. Suppose $0<u\leq \min{(1,p,q)}$ and $t>0$. Then there exists  $m\in L^{\infty}$ so that $\mathcal{K}_u^{0,t}[m]<\infty$, but $T_m$ is not bounded on $\dot{F}_p^{0,q}$.

\end{theorem}

\begin{theorem}\label{sharpthm1}
Suppose $p<\infty$ or $1\leq q$. Assume $r=\min{(p,q)}<1$, $u> r$, $t>0$, and $s\in\mathbb{R}$. Then there exists $m\in L^{\infty}$ so that
$\mathcal{K}_u^{s,t}[m]<\infty$, but $T_m$ is not bounded on $\dot{F}_p^{0,q}$ if one of the following conditions holds;
\begin{enumerate}
\item $s<d/r-d/u$,
\item $s=d/r-d/u$ and $t>r$.
\end{enumerate}
\end{theorem}

\begin{theorem}\label{sharpthm2}
Assume $1\leq p\leq\infty$, $1\leq q\leq \infty$, $1<u\leq \infty$, and $s\in\mathbb{R}$. Then there exists $m\in L^{\infty}$ so that
$\mathcal{K}_u^{s,t}[m]<\infty$, but $T_m$ is not bounded on $\dot{F}_p^{0,q}$ if one of the following conditions holds;
\begin{enumerate}
\item $s<d-d/u$,
\item $s=d-d/u$ and $p=1<q\leq \infty$,
\item $s=d-d/u$ and $1\leq q<p=\infty$,
\item $s=d-d/u$, $1<p<\infty$, and $\big|1/p-1/q \big|\geq 1-1/u$,
\item $s=d-d/u$, $t>1$.
\end{enumerate}
\end{theorem}

\begin{remark}
Theorem \ref{sharpbesov} proves the sharpness of Theorem \ref{multiplierbesov}, Theorem \ref{sharpthm0} does that of Theorem \ref{multipliertheorem0} and \ref{multipliertheorem2}, and Theorem \ref{sharpthm2} does that of Theorem \ref{multipliertheorem3}.
The sharpness of Theorem \ref{multipliertheorem1} is obtained from Theorem \ref{sharpthm1} except the case $p=\infty$ and $q<1$.
One does not have a conclusion for $p=\infty$ and $q<1$.

\end{remark}

\section{\textbf{Use of weighted Herz spaces in a limiting case}}\label{weight}

We are mainly interested in  the endpoint case, but when $s=0$, 
according to Theorem \ref{sharpthm0},  there exists a multiplier $m$ so that $\mathcal{K}_u^{0,u}[m]<\infty$ for $0<u\leq\min{(1,p,q)}$ and $T_m$ is not bounded on $\dot{F}_p^{0,q}$, $p\not= q$.
In this section we provide additional results for multipliers by replacing the condition $\mathcal{K}_u^{0,u}[m]<\infty$ by slightly stronger condition on $m$. \\

 Suppose that $w:\{0,1,2,\dots\}\to [1,\infty)$ satisfies $1\leq w(l)\leq w(l+1)<\infty$. For $0<u\leq \infty$ let ${K}_u^{0,u}(w)$ be the collection of all $f\in L^u_{loc}$ for which 
\begin{eqnarray*}
\Vert f\Vert_{{K}_u^{0,u}(w)}:=\Big( \sum_{l=0}^{\infty}{\Vert f\Vert_{L^u(B_l)}^u w(l)^u}\Big)^{1/u}<\infty.
\end{eqnarray*} 
Let
\begin{eqnarray*}
\mathcal{K}_u^{0,u}(w)[m]:=\sup_{k\in\mathbb{Z}}{\big\Vert \big( m(2^k\cdot)\varphi\big)^{\vee} \big\Vert_{{K}_u^{0,u}(w)}},
\end{eqnarray*}
and \begin{eqnarray*}
\mathcal{B}_{p,q}(w):=\Big(\sum_{l=0}^{\infty}w(l)^{-\frac{1}{|1/p-1/q|}}\Big)^{|1/p-1/q|}<\infty.
\end{eqnarray*} 
Then Baernstein and Sawyer \cite{Ba_Sa} proved that
if $\mathcal{B}_{1,2}(w)<\infty$ and $\mathcal{K}_1^{0,1}(w)[m]<\infty$
then \begin{eqnarray*}
\big\Vert T_mf\big\Vert_{H^1}\lesssim \mathcal{B}_{1,2}(w)\mathcal{K}_1^{0,1}(w)[m]\Vert f\Vert_{H^1}.
\end{eqnarray*}
Seeger \cite{Se2} generalized this result to $\dot{F}_p^{0,q}$ for $1\leq p, q\leq \infty$ and $p\not= q$, asserting that for $1\leq p,q\leq \infty$, $p\not= q$,  if  $\mathcal{B}_{p,q}(w)<\infty$ and $\mathcal{K}_1^{0,1}(w)[m]<\infty$ then
\begin{eqnarray*}
\Vert T_m\Vert_{\dot{F}_p^{s,q}\to \dot{F}_p^{s,q}}\lesssim \mathcal{B}_{p,q}(w)\mathcal{K}_1^{0,1}(w)[m]\Vert f\Vert_{\dot{F}_p^{0,q}}.
\end{eqnarray*} 

We improve these results by replacing $\mathcal{K}_1^{0,1}(w)[m]<\infty$ by $\mathcal{K}_u^{0,u}(w)[m]<\infty$ for $0<u\leq \min{(1,p,q)}$ and extending them to $0<p,q\leq \infty$.

\begin{theorem}\label{weightmultiplier}
Suppose $0<p,q\leq \infty$, $p\not= q$, and $0<u\leq \min{(1,p,q)}$. Let $w$ be an increasing positive function such that $\mathcal{B}_{p,q}(w)<\infty$. If $m\in L^{\infty}$ satisfies $\mathcal{K}_u^{0,u}(w)[m]<\infty$, then $T_m$ is bounded on $\dot{F}_p^{0,q}$. Moreover, in the case 
\begin{eqnarray*}
\big\Vert T_mf\big\Vert_{\dot{F}_p^{0,q}}\lesssim \mathcal{B}_{p,q}(w)\mathcal{K}_u^{0,u}(w)[m]\Vert f\Vert_{\dot{F}_p^{0,q}}.
\end{eqnarray*}
\end{theorem}

\begin{remark}
Theorem \ref{weightmultiplier} is sharp in the sense that the power $-\frac{1}{|1/p-1/q|}$ in the condition $\mathcal{B}_{p,q}(w)<\infty$ cannot be improved. Indeed, we shall prove the following theorem.
\end{remark}
\begin{theorem}\label{sharptheorem1}
Suppose $0\leq p,q\leq\infty$, $p\not= q$, and $t>0$.
Then there exists an increasing nonnegative sequence $\{w(k)\}$ and $m\in L^{\infty}$ so that
$\sum_{k=0}^{\infty}{w(k)^{-s}}<\infty$ for all $s>\frac{1}{|1/p-1/q|}$,
 $\mathcal{K}_{u}^{0,t}(w)[m]<\infty$, but $T_m$ is not bounded on $\dot{F}_p^{0,q}$.

 \end{theorem}

\section{\textbf{ Preliminary }}\label{prelim}

Let $\mathcal{D}_k$ be  the subset of $\mathcal{D}$ consisting of the cubes with side length $2^{-k}$ and for each $Q\in \mathcal{D}$ let  $\chi_Q$ stand for the characteristic function of $Q$.
For $r>0$ let $\mathcal{E}(r)$ be the space of tempered distributions whose Fourier transforms are supported in $\{\xi:|\xi|\leq 2r\}$.

\subsection{Maximal inequalities on $F$-space}

Let $\mathcal{M}$ denote the Hardy-Littlewood maximal operator and for $0<t<\infty$ let $\mathcal{M}_tf:=\big(  \mathcal{M}(|f|^t) \big)^{1/t}$. 
Then Fefferman-Stein's vector valued maximal inequality says that
for $0<r<p,q<\infty$
\begin{eqnarray}\label{hlmax}
\Big\Vert  \Big(\sum_{k}{(\mathcal{M}_{r}f_k)^q}\Big)^{1/{q}} \Big\Vert_{L^p} \lesssim  \Big\Vert \Big( \sum_{k}{|f_k|^q}  \Big)^{1/{q}}  \Big\Vert_{L^p}.
\end{eqnarray} 
Note that (\ref{hlmax}) also holds when $q=\infty$.

A crucial tool in theory of function space is a maximal operator introduced by Peetre \cite{Pe}.
For $k\in\mathbb{Z}$ and $\sigma>0$ define  \begin{eqnarray*}
\mathfrak{M}_{\sigma,2^k}f(x):=\sup_{y\in\mathbb{R}^d}{\frac{|f(x-y)|}{(1+2^k|y|)^{\sigma}}}.
\end{eqnarray*}
As shown in \cite{Pe} one has the majorization \begin{eqnarray*}
\mathfrak{M}_{d/r,2^k}f(x)\lesssim \mathcal{M}_rf(x), 
\end{eqnarray*}  provided that $f\in \mathcal{E}(2^k)$. 
Then via (\ref{hlmax}) the following maximal inequality holds.
Suppose $0<p<\infty$ and $0<q\leq \infty$. Then for  $f_k\in \mathcal{E}(2^k)$,
\begin{eqnarray}\label{max}
\Big\Vert  \Big(\sum_{k}{(\mathfrak{M}_{d/r,2^k}f_k)^q}\Big)^{1/{q}} \Big\Vert_{L^p} \lesssim  \Big\Vert \Big( \sum_{k}{|f_k|^q}  \Big)^{1/{q}}  \Big\Vert_{L^p} ~\text{for}~ r<\min{\big\{p,q\big\}}.
\end{eqnarray}

For $\epsilon\geq0$, $r>0$, and $k\in\mathbb{Z}$, we now introduce a variant of Hardy-Littlewood maximal operators $\mathcal{M}_{r}^{k,\epsilon}$, defined by
\begin{eqnarray*}
\mathcal{M}_r^{k,\epsilon}f(x)&:=&\sup_{2^kl(Q)\leq 1,x\in Q}{\Big( \frac{1}{|Q|}\int_{Q}{|f(x)|^r}dx  \Big)^{1/r}}\\
  &&+\sup_{2^{k}l(Q)>1,x\in Q}{\big(2^kl(Q)\big)^{-\epsilon}\Big( \frac{1}{|Q|}\int_Q{|f(x)|^r}dx  \Big)^{1/r}}. \nonumber
\end{eqnarray*}
Then it is proved in \cite{Park2} that for $r<t$ and  $f\in\mathcal{E}(2^k)$
\begin{eqnarray*}
\mathfrak{M}_{d/r,2^k}f(x)\lesssim \mathcal{M}_t^{k,d/r-d/t}f(x)\lesssim \mathcal{M}_tf(x).
\end{eqnarray*} 
Furthermore the following maximal inequalities hold.
\begin{lemma}\cite{Park2}\label{maximal1}
Let $0<r<q<\infty$ and $\epsilon>0$. For $k\in\mathbb{Z}$ let $f_k\in\mathcal{E}(2^{k+h})$ for some $h\in\mathbb{Z}$. Let $P\in\mathcal{D}$ and $l(P)=2^{-\mu}$.
Then 
\begin{eqnarray*}
\sup_{P\in\mathcal{D}_{\mu}}{\Big(  \frac{1}{|P|}\int_P{  \sum_{k=\mu}^{\infty}{  \big( \mathcal{M}_r^{k,\epsilon}f_k(x)  \big)^{q}      }    }dx  \Big)^{1/q}} \lesssim_h \sup_{R\in\mathcal{D}_{\mu}}{\Big(  \frac{1}{|R|}\int_R{  \sum_{k=\mu}^{\infty}{   |f_k(x)|^q   }    }dx  \Big)^{1/q}}. 
\end{eqnarray*}
Here, the implicit constant of the inequality is independent of $\mu$.
\end{lemma}

\begin{lemma}\cite{Park2}\label{maximal2}
Let $0<r<q<\infty$.
For $k\in\mathbb{Z}$ let $f_k\in\mathcal{E}(2^{k+h})$ for some $h\in\mathbb{Z}$. Let $P\in\mathcal{D}$ and $l(P)=2^{-\mu}$.
Then 
\begin{eqnarray*}
\sup_{P\in\mathcal{D}_{\mu}}{\Big(  \frac{1}{|P|}\int_P{  \sum_{k=\mu}^{\infty}{  \big(\mathfrak{M}_{d/r,2^k}f_k(x) \big)^{q}      }    }dx  \Big)^{1/q}} \label{boundee}\lesssim \sup_{R\in\mathcal{D}_{\mu}}{\Big(  \frac{1}{|R|}\int_R{  \sum_{k=\mu}^{\infty}{   |f_k(x)|^q   }    }dx  \Big)^{1/q}}. \label{bounder}
\end{eqnarray*}
Here, the implicit constant of the inequality is independent of $\mu$.
\end{lemma}

As an application of Lemma \ref{maximal2}, for $\mu\in\mathbb{Z}$, $0<q_1<q_2<\infty$, and $\mathbf{f}:=\{f_k\}_{k\in\mathbb{Z}}$ one has
\begin{eqnarray*}
\mathcal{V}_{\mu,q_2}[\mathbf{f}]\lesssim \mathcal{V}_{\mu,q_1}[\mathbf{f}],
\end{eqnarray*} provided that each $f_k$ is defined as in Lemma \ref{maximal2}, where
\begin{eqnarray*}
\mathcal{V}_{\mu,q}[\mathbf{f}]:= \sup_{P\in\mathcal{D},l(P)\leq 2^{-\mu}}{\Big(  \frac{1}{|P|}\int_P{  \sum_{k=-\log_2{l(P)}}^{\infty}{   |f_k(x)|^q   }    }dx  \Big)^{1/q}}. \end{eqnarray*}
This definitely implies the embedding $\dot{F}_{\infty}^{0,q_1}\hookrightarrow \dot{F}_{\infty}^{0,q_2}$ for $0<q_1<q_2\leq \infty$.
See \cite{Park2} for more details.

\subsection{$\varphi$-transform of $\dot{F}$-spaces}\label{decomposition}( \cite{Fr_Ja}, \cite{Fr_Ja1}, \cite{Fr_Ja2} )
We define sequence spaces $\dot{f}_p^{\alpha,q}$ associated with $\dot{F}_p^{\alpha,q}$ as the family of 
 sequences of complex numbers $b=\{b_Q\}_{Q\in\mathcal{D}}$ for which
 $$\Vert b \Vert_{\dot{f}_p^{\alpha,q}}=\big\Vert  g^{\alpha,q}(b)  \big\Vert_{L^p}<\infty$$
 where
$$g^{\alpha,q}(b)(x)=\Big(\sum_{Q\in\mathcal{D}}{\big(|Q|^{-\alpha /{d}-1/2}|b_Q|\chi_Q(x)\big)^q}\Big)^{1/q}.$$
Furthermore, for $c>0$ let  $\vartheta$ and $\widetilde{\vartheta}$ be Schwartz functions satisfying 
\begin{eqnarray*}
Supp(\widehat{\vartheta}), Supp(\widehat{\widetilde{\vartheta}})\subset \{\xi : 1/{2}\leq |\xi|\leq 2\}
\end{eqnarray*}
\begin{eqnarray*}
|\widehat{\vartheta}(\xi)|, |\widehat{\widetilde{\vartheta}}(\xi)| \geq c>0 ~\text{for}~ 3/4\leq |\xi|\leq 5/3
\end{eqnarray*}
 \begin{eqnarray*}
 \sum_{k\in\mathbb{Z}}{\overline{\widetilde{\vartheta}_k(\xi)}\vartheta_k(\xi)}=1 , \quad \xi\not=0
 \end{eqnarray*} where $\vartheta_k(x)=2^{kd}\vartheta(2^kx)$  and $\widetilde{\vartheta}_k(x)=2^{kd}\widetilde{\vartheta}(2^kx)$ for $k\in\mathbb{Z}$.
Then the (quasi-)norms in  $\dot{F}_p^{\alpha,q}$ can be characterized by the $\dot{f}_p^{\alpha,q}$ (quasi-)norms as follows.
Suppose $0<p<\infty$, $0<q\leq\infty$, $\alpha\in\mathbb{R}$. 
 Every $f\in \dot{F}_p^{\alpha,q}$ can be decomposed as 
\begin{eqnarray}\label{decomposition1}
f(x)=\sum_{Q\in\mathcal{D}}{v_Q\vartheta^Q(x)} 
\end{eqnarray} where $x_Q$ stands for the lower left corner of $Q$, $\vartheta^Q(x):=|Q|^{1/2}\vartheta_k(x-x_Q)$ for $Q\in\mathcal{D}_k$ and $v_Q:=\langle f,\widetilde{\vartheta}^Q\rangle$.
Moreover, in the case one has \begin{eqnarray*}
\big\Vert v   \big\Vert_{\dot{f}_p^{\alpha,q}} \lesssim \big\Vert  f  \big\Vert_{\dot{F}_p^{\alpha,q}}.
\end{eqnarray*}
The converse estimate also holds.  For any sequence $v=\{v_Q\}_{Q\in\mathcal{D}}$ of complex numbers satisfying $\big\Vert  v \big\Vert_{\dot{f}_p^{\alpha,q}}<\infty$,  $$f(x):=\sum_{{Q\in\mathcal{D}}}{v_Q\vartheta^Q(x)}$$ belongs to $\dot{F}_p^{\alpha,q}$ and to be specific
\begin{eqnarray}\label{decomposition2}
\big\Vert  f  \big\Vert_{\dot{F}_p^{\alpha,q}} \lesssim \big\Vert  v \big\Vert_{\dot{f}_p^{\alpha,q}}.
\end{eqnarray}

\subsection{$\infty$-atoms for $\dot{f}_p^{\alpha,q}$ ( \cite[Chapter 7]{Fr_Ja}, \cite{Fr_Ja2}, \cite[Chap.6.6.3]{Gr} ) }

Let $0<p\leq 1$, $0< q\leq \infty$, and $\alpha\in\mathbb{R}$. A sequence of complex numbers $a=\{a_Q\}_{Q\in\mathcal{D}}$ is called an $\infty$-atom for $\dot{f}_p^{\alpha,q}$ if there exists a dyadic cube $Q_0$ such that 
\begin{eqnarray*}
a_Q=0 \quad \text{if}\quad Q \not\subset Q_0
\end{eqnarray*}
 and \begin{eqnarray}\label{infdef}
\big\Vert  g^{\alpha,q}(a)  \big\Vert_{L^{\infty}}\leq |Q_0|^{-{1}/{p}}.
\end{eqnarray}

Then the following atomic decomposition of $\dot{f}_p^{\alpha,q}$, which is analogous to the atomic decomposition of $H^p$, holds.
\begin{lemma}\label{decomhardy}
Suppose $0<p\leq 1$, $p\leq q\leq\infty$, and $b=\{b_Q\}_{Q\in\mathcal{D}}\in \dot{f}_p^{\alpha,q}$. Then there exist $C_{d,p,q}>0$, a sequence of scalars $\{\lambda_j\}$, and a sequence of $\infty$-atoms $a_j=\{a_{j,Q}\}_{{Q\in\mathcal{D}}}$ for $\dot{f}_p^{\alpha,q}$ so that $$b=\{b_Q\}=\sum_{j=1}^{\infty}{\lambda_j\{a_{j,Q}\}}=\sum_{j=1}^{\infty}{\lambda_j a_j},$$ and $$\Big(\sum_{j=1}^{\infty}{|\lambda_j|^p}\Big)^{{1}/{p}}\leq C_{d,p,q}\big\Vert   b\big\Vert_{\dot{f}_{p}^{\alpha,q}}.$$
Moreoever, it follows that \begin{eqnarray*}\label{equismall}
\big\Vert  b  \big\Vert_{\dot{f}_p^{\alpha,q}}\approx \inf{\Big\{ \Big(\sum_{j=1}^{\infty}{|\lambda_j|^p}\Big)^{{1}/{p}}   :  b=\sum_{j=1}^{\infty}{\lambda_j a_j} ,~ a_j ~\text{is an $\infty$-atom for $\dot{f}_p^{\alpha,q}$}    \Big\}}.
\end{eqnarray*}

\end{lemma}

\subsection{Duality in $\dot{F}_p^{\alpha,q}$}
Let $$S_{\infty}:=\Big\{f\in S: \int{x^{\gamma}f(x)}dx=0 ~\text{for all multi-indices}~\gamma\Big\}.$$
Then $S_{\infty}$ is a subspace of $S$ that inherits the same topology as $S$ and whose dual is $S'/\mathcal{P}$.
Moreover, $S_{\infty}$  is dense in $\dot{F}_p^{\alpha,q}$ if $0<p,q<\infty$.
It is known in \cite[Theorem 5.13, Remark 5.14]{Fr_Ja} that for $1\leq p<\infty$, $1\leq q<\infty$, and $\alpha\in\mathbb{R}$,
\begin{eqnarray}\label{dual}
(\dot{F}_p^{\alpha,q})' = \dot{F}_{p'}^{-\alpha,q'}
\end{eqnarray} where $1/q+1/{q'}=1$  and similarly for $p'$.
In addition, in the case \begin{eqnarray*}
\Vert f\Vert_{\dot{F}_p^{\alpha,q}}\approx\sup{\big\{|\langle f,g\rangle|:g\in S_{\infty}~\text{with}~\Vert g\Vert_{\dot{F}_{p'}^{-\alpha,q'}}\leq 1\big\}}.
\end{eqnarray*}

\subsection{Interpolation theory}
We refer to \cite[Chapter 6]{Fr_Ja} for real interpolation and \cite[Chapter 8]{Fr_Ja} for complex interpolation.

By using Peetre's real interpolation method, so called $K$-method,
it is known in \cite[Chapter 6]{Fr_Ja} that for fixed $s\in \mathbb{R}$ and $0<q\leq \infty$, one has
\begin{eqnarray*}
(\dot{F}_{p_0}^{s,q},\dot{F}_{p_1}^{s,q})_{\theta,p}=\dot{F}_p^{s,q}, \quad \text{if}~1/p=(1-\theta)/p_0+\theta/p_1 
\end{eqnarray*} where $0<p_0<p<p_1\leq \infty$ and $0<\theta<1$.
This implies the estimate
\begin{eqnarray}\label{realinter}
\Vert T\Vert_{\dot{F}_{p}^{0,q}\to\dot{F}_{p}^{0,q}}\lesssim \Vert T\Vert_{\dot{F}_{p_0}^{0,q}\to\dot{F}_{p_0}^{0,q}}^{1-\theta}\Vert T\Vert_{\dot{F}_{p_1}^{0,q}\to\dot{F}_{p_1}^{0,q}}^{\theta}.
\end{eqnarray}

Now we apply a complex interpolation method for a Banach couple.
Let $s_0,s_1\in\mathbb{R}$, $1\leq p_0,q_0<\infty$ and $1\leq p_1,q_1\leq \infty$. Let $0<\theta<1$ and suppose
\begin{eqnarray*}
s=(1-\theta)s_0+\theta s_1, ~1/p=(1-\theta)/p_0+\theta/p_1, ~1/q=(1-\theta)/q_0+\theta/q_1.
\end{eqnarray*} Then
\begin{eqnarray*}
[\dot{F}_{p_0}^{s_0,q_0},\dot{F}_{p_1}^{s_1,q_1}]_{\theta}=\dot{F}_{p}^{s,q}
\end{eqnarray*}
Then the following interpolation holds.
\begin{eqnarray}\label{complexinter}
\Vert T\Vert_{\dot{F}_{p}^{0,q}\to\dot{F}_{p}^{0,q}}\lesssim \Vert T\Vert_{\dot{F}_{p_0}^{0,q_0}\to\dot{F}_{p_0}^{0,q_0}}^{1-\theta}\Vert T\Vert_{\dot{F}_{p_1}^{0,q_1}\to\dot{F}_{p_1}^{0,q_1}}^{\theta}.
\end{eqnarray}

\section{\textbf{Proof of Theorem \ref{multiplierbesov} $-$ \ref{multipliertheorem3}}}\label{mmultiplier}

We start the proof by introducing some key lemmas.
Nikolskii's inequalitiy \cite{Ni} says that for $0<p< q\leq \infty$, if  $f\in S'\cap L^p$ has Fourier transform which is compactly supported in a ball of radius $2^k$ then 
\begin{eqnarray}\label{classicalnikol}
\Vert f\Vert_{L^q}\lesssim 2^{kd(1/p-1/q)}\Vert f\Vert_{L^p}.
\end{eqnarray} 
We have an improvement of this estimate.
\begin{lemma}\label{nikol}
Let $0<p<q\leq \infty$ and $f\in S'\cap K_p^{0,q}$ with $Supp(\widehat{f})\subset \{\xi:|\xi|\lesssim 2^k\}$. Then
\begin{eqnarray*}
\Vert f\Vert_{L^q}\lesssim 2^{kd(1/p-1/q)}\Vert f\Vert_{K_{p}^{0,q}}.
\end{eqnarray*}
\end{lemma}
\begin{remark}
This lemma clearly implies (\ref{classicalnikol}) because $l^p\hookrightarrow l^q$, $L^p=K_p^{0,p}$, and $L_q=K_q^{0,q}$.
\end{remark}
\begin{proof}[Proof of Lemma \ref{nikol}]
Suppose $0<\Vert f\Vert_{L^q}<\infty$.
Let $h_d$ be the integer satisfying $\log_2{\sqrt{d}}<h_d\leq \log_2{\sqrt{d}}+1$ and let $P_0:=[-2,2]^d$ and  $P_l:=[-2^{l+1},2^{l+1}]^d\setminus [-2^{l-h_d},2^{l-h_d}]^d$ for $l\geq 1$. Then $P_l$ contains $B_l$ for all $l\geq 0$.
Now
\begin{eqnarray*}
\Vert f\Vert_{L^q}&=&\Big( \sum_{l=0}^{\infty}{\int_{P_l}{\big| f(x)\chi_{B_l}(x)\big|^q}dx}\Big)^{1/q}\\
  &=&\Big(\sum_{l=0}^{\infty}{\sum_{Q\subset P_l,Q\in\mathcal{D}_k}{\int_{Q}{\big|f(x)\chi_{B_l}(x) \big|^p}dx \big(\sup_{y\in Q}{|f(y)|} \big)^{q(1-p/q)}}} \Big)^{1/q}.
\end{eqnarray*}
It should be observed that for any $\sigma>0$ and $Q\in\mathcal{D}_k$ 
\begin{eqnarray}\label{inftykey}
\sup_{y\in Q}{|f(y)|}\lesssim_{\sigma} \inf_{y\in Q}{\mathfrak{M}_{\sigma,2^k}f(y)}.
\end{eqnarray}
As a consequence, one obtains \begin{eqnarray*}
\big( \sup_{y\in Q}{|f(y)|}\big)^{q(1-p/q)}&\lesssim_{\sigma}& 2^{kd(1-p/q)}\Big(\int_Q{\big(\mathfrak{M}_{\sigma,2^k}f(y) \big)^q}dy \Big)^{1-p/q} 
\end{eqnarray*} and thus by H\"older's inequality with $q/p>1$
\begin{eqnarray*}
\Vert f\Vert_{L^q}&\lesssim_{\sigma}&2^{kd/q(1-p/q)}\Big(\sum_{l=0}^{\infty}{\Big( \int_{P_l}{\big( \mathfrak{M}_{\sigma,2^k}f(y)\big)^q}dy\Big)^{1-p/q}\int_{B_l}{|f(x)|^p}dx} \Big)^{1/q}\\
&\lesssim&2^{kd/q(1-p/q)}\Vert f\Vert_{K_p^{0,q}}^{p/q}\Big(\sum_{l=0}^{\infty}{\int_{P_l}{\big(\mathfrak{M}_{\sigma,2^k}f(y) \big)^q}dy} \Big)^{1/q(1-p/q)}\\
&\lesssim&2^{kd/q(1-p/q)}\Vert f\Vert_{K_p^{0,q}}^{p/q}\big\Vert \mathfrak{M}_{\sigma,2^k}f\big\Vert_{L^q}^{1-p/q}.
\end{eqnarray*}
Then we choose $\sigma>d/q$ and  apply $(\ref{max})$ since $Supp(\widehat{f})\subset \{\xi:|\xi|\lesssim 2^k\}$. By dividing both sides by $\Vert f\Vert_{L^q}^{1-p/q}$ one completes the proof.

\end{proof}

We recall that $\varphi$ is a cutoff function so that $0\leq \varphi\leq 1$, $Supp(\varphi)\subset \{\xi:1/4\leq |\xi|\leq 4\}$, and $\varphi=1$ on $\{\xi:1/2\leq |\xi|\leq 2\}$.
 For $k\in\mathbb{Z}$
 let 
 \begin{eqnarray}\label{mkdefinition}
 m_k:=m\varphi(\cdot/2^k).
 \end{eqnarray}
Then one has $$T_mf(x)=\sum_{k\in\mathbb{Z}}{{\Pi}_kT_{m}f(x)}=\sum_{k\in\mathbb{Z}}{m_k^{\vee}\ast (\Pi_kf)(x)}=\sum_{k=0}^{\infty}{T_{m_k} (\Pi_kf)(x)}$$
and it follows from Lemma \ref{nikol} that for any $0<u\leq 1$ \begin{eqnarray}\label{boundm}
\Vert m_k^{\vee}\Vert_{L^1}=\big\Vert \big(m(2^k\cdot)\varphi \big)^{\vee}\big\Vert_{L^1}\lesssim \mathcal{K}_{u}^{0,1}[m].
\end{eqnarray} 
Similarly, it is also clear that for $0<u\leq p<1$
\begin{eqnarray}\label{boundmp}
\Vert m_k^{\vee}\Vert_{L^p}&=&2^{-kd(1/p-1)}\big\Vert \big(m(2^k\cdot)\varphi \big)^{\vee} \big\Vert_{L^p}\lesssim 2^{-kd(1/p-1)}\mathcal{K}_u^{0,p}[m].
\end{eqnarray}
Then the following lemma holds as a corollary of (\ref{boundm}) and (\ref{boundmp}).
\begin{lemma}\label{less}
Let $0<p\leq\infty$, $k\in \mathbb{Z}$, and $m_k$ be as in $(\ref{mkdefinition})$.
\begin{enumerate}
\item 
For $0<u\leq 1\leq p\leq \infty$ 
\begin{eqnarray*}
\big\Vert T_{m_k}\big\Vert_{L^p\to L^p}\lesssim \mathcal{K}_{u}^{0,1}[m]
\end{eqnarray*} uniformly in $k$.

\item For $0<u\leq p< 1$ and $A>0$ if $f_k\in\mathcal{E}(A2^k)$ then 
\begin{eqnarray*}
\big\Vert T_{m_k}f_k\big\Vert_{L^p}\lesssim_A  \mathcal{K}_{u}^{0,p}[m]\big\Vert f_k\big\Vert_{L^p}
\end{eqnarray*} uniformly in $k$. 
\end{enumerate}

\end{lemma}
The proof of Lemma \ref{less} is quite simple. We apply Young's inequality for $1\leq p\leq \infty$ to obtain $\big\Vert T_{m_k}f\big\Vert_{L^p}\leq \Vert m_k^{\vee}\Vert_{L^1} \Vert f\Vert_{L^p}$ and then use (\ref{boundm}). If $0<p<1$ then Nikolskii's inequality (\ref{classicalnikol}) can be applied. Indeed, $\big\Vert T_{m_k}f_k\big\Vert_{L^p}\lesssim_A 2^{kd(1/p-1)}\Vert m_k^{\vee}\Vert_{L^p}\Vert f_k \Vert_{L^p}$ and finally (\ref{boundmp}) proves (2).\\

Note that Lemma \ref{less} gives the straightforward proof of the first statement in Theorem \ref{multiplierbesov} and by embedding $K_u^{d/r-d/u,r}\hookrightarrow K_r^{0,r}$ for $r<u$, the second one is immediate from the first one in the theorem.\\

From now on we assume $p\not= q$ and prove Theorem \ref{multipliertheorem0}$-$\ref{multipliertheorem3}.
We deal with the case $0<q<p\leq \infty$ and the case $0<p\leq 1$, $p<q\leq \infty$ because the results for $1<p<q\leq \infty$ follow from the duality argument (\ref{dual}).

When $0<p\leq 1$ and $p<q\leq \infty$ the proof is based on the method of $\varphi$-transform and $\infty$-atoms for $\dot{f}_p^{0,q}$.
By applying  (\ref{decomposition1}) and Lemma \ref{decomhardy}, $f\in \dot{F}_p^{0,q}$ can be decomposed as
$$f(x)=\sum_{Q\in\mathcal{D}}{b_Q\vartheta^Q(x)}=\sum_{j=1}^{\infty}{\lambda_j \sum_{Q\in\mathcal{D}}{ a_{j,Q}\vartheta^Q (x)  }}$$
 for some $\{b_Q\}_{Q\in\mathcal{D}}\in \dot{f}_p^{0,q}$, a sequence of scalars $\{\lambda_j\}$, and a sequence of $\infty$-atoms $\{a_{j,Q}\}$ for $\dot{f}_p^{0,q}$.
Then  by applying $l^p\hookrightarrow l^1$ and  Minkowski's inequality with $q/p>1$ as in \cite{Park},  one has
\begin{eqnarray}\label{sup1}
\big\Vert   T_mf \big\Vert_{\dot{F}_p^{0,q}} &\lesssim& \big( \sum_{j=1}^{\infty}{|\lambda_j|^p}\big)^{1/p}\sup_{l\geq 1}{\Big\Vert \Big( \sum_{k\in\mathbb{Z}}{\big| m_k^{\vee}\ast\big(\sum_{Q\in\mathcal{D}_k}{a_{l,Q}\vartheta^{Q}} \big)\big|^q}\Big)^{1/q} \Big\Vert_{L^p}}.
\end{eqnarray} 
Since $$\Big( \sum_{j=1}^{\infty}{|\lambda_j|^p} \Big)^{{1}/{p}}\lesssim \Vert f \Vert_{\dot{F}_p^{0,q}},$$ 
the proof can be reduced to obtaining the desired bound for the supremum of the right hand side in (\ref{sup1}).

When $q<p$ we shall employ  the sharp maximal function estimates  in \cite{Se2}.
 For a sequence of Schwartz functions $\{f_k\}_{k\in\mathbb{Z}}$ let $$\mathcal{N}_{q}^{\sharp}\big(\{f_k\}\big)(x):=\sup_{P:x\in P\in\mathcal{D}}{\Big(\dfrac{1}{|P|}\int_P{\sum_{k=-\log_2{l(P)}}^{\infty} |f_k(y)|^q}dy \Big)^{1/q}}.$$
 Then it is known in \cite{Se2} that
 for $0<q<p\leq \infty$
 \begin{eqnarray}\label{lemmasharp}
 \big\Vert \mathcal{N}_q^{\sharp}\big( \{\Pi_kf\}\big)\big\Vert_{L^p}\approx \Vert f \Vert_{\dot{F}_p^{0,q}}.
 \end{eqnarray}

\subsection{Proof of Theorem \ref{multipliertheorem0}}

Suppose $0<u\leq r=\min{(p,q)}<1$ and $\mathcal{K}_{u}^{s,u}[m]<\infty$.
Due to a proper embedding with $s>0$  one may assume $t=u$.

\subsubsection{The case $0<p<1$ and $p<q\leq \infty$ }
Suppose $0<u\leq p$ and $s>0$.
Let $Q_0$ be any dyadic cubes with side length $2^{-\mu}$, $\mu\in\mathbb{Z}$, and $a_Q$ be an $\infty$-atom for $\dot{f}_p^{0,q}$ with $Q_0$. For  each $k\in\mathbb{Z}$ let 
\begin{eqnarray}\label{rq}
 A_{Q_0,k}(x):=\sum_{{Q\in\mathcal{D}_k, Q\subset Q_0}}{a_Q\vartheta^Q(x)}.
 \end{eqnarray} 
The condition $Q\subset Q_0$ in the sum ensures that $A_{Q_0,k}$ vanishes unless $\mu\leq k$ and thus
 one obtains the desired result by showing
\begin{eqnarray}\label{gggoal}
\Big\Vert \Big( \sum_{k=\mu}^{\infty}{\big| T_{m_k}( A_{Q_0,k}) \big|^q}\Big)^{1/q}\Big\Vert_{L^p} \lesssim \mathcal{K}_{u}^{s,u}[m]\quad \text{uniformly in $Q_0$}.
\end{eqnarray} 

For $j\geq 1$ let \begin{eqnarray*}
\mathcal{G}_{Q_0,p,q}^{j}:=\Big(\int_{D_j}{\Big(\sum_{k=\mu}^{\infty}{\big| T_{m_k}(A_{Q_0,k})(x)\big|^q} \Big)^{p/q}}dx\Big)^{1/p}
\end{eqnarray*} where
 $$D_1:=\{x\in\mathbb{R}^d:|x-c_{Q_0}|\leq 2^{2-\mu}\sqrt{d}\},$$
$$D_j:=\{x\in\mathbb{R}^d:2^{j-\mu}\sqrt{d}<|x-c_{Q_0}|\leq 2^{j-\mu+1}\sqrt{d}\},\quad j\geq 2.$$ 
Then the left hand side of (\ref{gggoal}) is bounded by
$$\mathcal{G}_{Q_0,p,q}^{1}+\Big( \sum_{j=2}^{\infty}{\big( \mathcal{G}_{Q_0,p,q}^{j}\big)^p}\Big)
^{1/p}.$$

The estimation of $\mathcal{G}_{Q_0,p,q}^1$ is straightforward. By using H\"older's inequality with $q/p>1$, Lemma \ref{less}, (\ref{decomposition2}), and (\ref{infdef}) one obtains 
 \begin{eqnarray*}
 \mathcal{G}_{Q_0,p,q}^1&\lesssim&|Q_0|^{1/p-1/q}\Big(\sum_{k=\mu}^{\infty}{\big\Vert T_{m_k}( A_{Q_0,k})\big\Vert_{L^q}^q} \Big)^{1/q}\nonumber \\
 &\lesssim& \mathcal{K}_{u}^{0,\min{(1,q)}}[m]|Q_0|^{1/p-1/q} \Big(\sum_{k=\mu}^{\infty}{\Vert A_{Q_0,k}\Vert_{L^q}^q} \Big)^{1/q}\lesssim  \mathcal{K}_{u}^{0,u}[m].
 \end{eqnarray*} 
 The last inequality follows from the fact that
\begin{eqnarray*}
\sum_{k=\mu}^{\infty}{\Vert A_{Q_0,k}\Vert_{L^q}^q} &\lesssim&\sum_{k=\mu}^{\infty}\int_{Q_0}{\sum_{Q\in\mathcal{D}_k,Q\subset Q_0}{\big(|a_Q||Q|^{-1/2}\chi_Q(x) \big)^{q}}}dx\\
   &=&\int_{Q_0}\sum_{Q\in\mathcal{D}, Q\subset Q_0}{\big( |a_Q||Q|^{-1/2}\chi_Q(x)\big)^q}dx \lesssim|Q_0|^{1-q/p}.
\end{eqnarray*}

Now it remains to show
\begin{eqnarray}\label{edc}
\displaystyle\Big( \sum_{j=2}^{\infty}{\big( \mathcal{G}_{Q_0,p,q}^{j}\big)^p}\Big)
^{1/p}\lesssim \mathcal{K}_{u}^{s,u}[m].
\end{eqnarray} 
 By using Nikolskii's inequality (\ref{classicalnikol}) the left hand side of (\ref{edc}) is controlled by
\begin{eqnarray}\label{dcp1}
\Big(\sum_{j=2}^{\infty}{\int_{D_j}{\Big[ \sum_{k=\mu}^{\infty}{2^{kdq(1/u-1)}\big\Vert m_k^{\vee}A_{Q_0,k}(x-\cdot)\big\Vert_{L^u}^q }    \Big]^{p/q}}dx} \Big)^{1/p}.
\end{eqnarray}
Then one decompose (\ref{dcp1}) into three pieces by using
\begin{eqnarray}\label{dcp2}
\big\Vert m_k^{\vee}A_{Q_0,k}(x-\cdot)\big\Vert_{L^u}&\lesssim& \mathcal{L}_{u,j}^{low}[A_{Q_0,k}](x)+\mathcal{L}_{u,j}^{mid}[A_{Q_0,k}](x)+\mathcal{L}_{u,j}^{high}[A_{Q_0,k}](x)
\end{eqnarray}
where
\begin{eqnarray*}
\mathcal{L}_{u,j}^{low}[A_{Q_0,k}](x):=\Big(\sum_{l: 2^l\leq 2^{j+k-\mu-2}\sqrt{d}}{\big\Vert m_k^{\vee}\chi_{B_l}(2^k\cdot)A_{Q_0,k}(x-\cdot)\big\Vert_{L^u}^u}\Big)^{1/u},
\end{eqnarray*} 
\begin{eqnarray*}
\mathcal{L}_{u,j}^{mid}[A_{Q_0,k}](x):=\Big(\sum_{l: 2^{j+k-\mu-2}\sqrt{d} <2^l<2^{j+k-\mu+3}\sqrt{d}}{\big\Vert m_k^{\vee}\chi_{B_l}(2^k\cdot)A_{Q_0,k}(x-\cdot)\big\Vert_{L^u}^u}\Big)^{1/u},
\end{eqnarray*}
\begin{eqnarray*}
\mathcal{L}_{u,j}^{high}[A_{Q_0,k}](x):=\Big(\sum_{l:2^l\geq 2^{j+k-\mu+3}\sqrt{d}}{\big\Vert m_k^{\vee}\chi_{B_l}(2^k\cdot)A_{Q_0,k}(x-\cdot)\big\Vert_{L^u}^u}\Big)^{1/u}.
\end{eqnarray*}

It should be observed that from (\ref{infdef}) \begin{eqnarray}\label{observe}
|a_Q|\leq |Q|^{1/2}|Q_0|^{-1/p}
\end{eqnarray} and thus for
 arbitrary $M>0$
 \begin{eqnarray*}
|A_{Q_0,k}(x-y)|&\lesssim_M&\sum_{Q\in\mathcal{D}_k,Q\subset Q_0}{|a_Q||Q|^{-1/2}\frac{1}{\big(1+2^k|x-y-x_Q| \big)^M}}\\
                         &\leq&|Q_0|^{-1/p}\sum_{Q\in\mathcal{D}_k,Q\subset Q_0}{\frac{1}{(1+2^k|x-y-x_Q|)^M}}.
\end{eqnarray*}
If $2^l\leq 2^{j+k-\mu-2}\sqrt{d}$ or $2^l\geq 2^{j+k-\mu+3}\sqrt{d}$ then one has $|x-y-c_{Q_0}|\geq 2^{j-\mu-2}\sqrt{d}$ for
$x\in D_j$ and $2^ky\in B_{l}$.
Accordingly, in this case
 \begin{eqnarray*}
|A_{Q_0,k}(x-y)|       &\lesssim&\big( 2^{k+j}l(Q_0)\big)^{-M}|Q_0|^{1-1/p}|Q|^{-1}
\end{eqnarray*}
and this proves
\begin{eqnarray}\label{jkjk}
&&\mathcal{L}_{u,j}^{low}[A_{Q_0,k}](x),\mathcal{L}_{u,j}^{high}[A_{Q_0,k}](x)\nonumber\\
&\lesssim_M&\Big(\sum_{l=k}^{\infty}{\big\Vert m_k^{\vee}\chi_{B_l}(2^k\cdot) \big\Vert_{L^u}^u}\Big)^{1/u}\big( 2^{k+j-\mu}\big)^{-M}|Q_0|^{1-1/p}|Q|^{-1}\label{xlxl}\nonumber\\
     &\lesssim&2^{-kd(1/u-1)}\mathcal{K}_{u}^{0,u}[m]\big( 2^{k+j-\mu}\big)^{-M}|Q_0|^{1-1/p}|Q|^{-1}.
\end{eqnarray}
By choosing $M>d/p$, the terms corresponding to $\mathcal{L}^{low}_{u,j}[A_{Q_0,k}]$ and $\mathcal{L}_{u,j}^{high}[A_{Q_0,k}]$ are less than a constant times $\mathcal{K}_{u}^{0,u}[m]$.

To complete the estimation of (\ref{dcp1}) one has to deal with
\begin{eqnarray*}
&&\Big( \sum_{j=2}^{\infty}{\int_{D_j}{\Big[ \sum_{k=\mu}^{\infty}{2^{kdq(1/u-1)}\big(\mathcal{L}_{u,j}^{mid}[A_{Q_0,k}](x) \big)^q}\Big]^{p/q}}dx}\Big)^{1/p}.
\end{eqnarray*} 
By using $l^p\hookrightarrow l^q$ this is dominated by
\begin{eqnarray}\label{middle}
\Big(\sum_{k=\mu}^{\infty}{2^{kdp(1/u-1)}\sum_{j=2}^{\infty}{\big\Vert \mathcal{L}_{u,j}^{mid}[A_{Q_0,k}]\big\Vert_{L^p(D_j)}^p}} \Big)^{1/p}.
\end{eqnarray} 
By using Minkowski's inequality if $u<p$ one obtains
\begin{eqnarray*}
\big\Vert \mathcal{L}_{u,j}^{mid}[A_{Q_0,k}]\big\Vert_{L^p(D_j)}&\lesssim&\Big(\sum_{l: 2^l\approx  2^{j+k-\mu}}{\big\Vert m_k^{\vee}\chi_{B_l}(2^k\cdot)\big\Vert_{L^u}^{u}\Vert A_{Q_0,k}\Vert_{L^p}^u} \Big)^{1/u}
\end{eqnarray*}
and this can be bounded by  $2^{-kd(1/u-1)}2^{-s(j+k-\mu)}\mathcal{K}_{u}^{s,u}[m]$, using the fact $\Vert A_{Q_0,k}\Vert_{L^p}\lesssim 1$ and $m_k^{\vee}=2^{kd}\big( m(2^k\cdot)\varphi\big)^{\vee}(2^k\cdot)$.
This proves $(\ref{middle})\lesssim \mathcal{K}_{u}^{s,u}[m]$.

\subsubsection{ The case $0<q<1$ and $q<p< \infty$ }
Suppose $0<u\leq q$ and $s>0$.
We deal with just the case $d/s<p<\infty$ because other cases follow by the real interpolation method (\ref{realinter}) with the case $p=q$.
Suppose $d/s<p<\infty$.
By (\ref{lemmasharp}) one has
\begin{eqnarray}\label{q<p}
\big\Vert T_mf \big\Vert_{\dot{F}_p^{0,q}}\approx \big\Vert  \mathcal{N}_q^{\sharp}(\{T_{m_k}({\Pi_k}f)\})\big\Vert_{L^p}.
\end{eqnarray}
Fix $x\in\mathbb{R}^d$ and a dyadic cube $P$ containing $x$. By Nikolskii's inequality (\ref{classicalnikol}) one obtains
\begin{eqnarray}\label{decompr}
&&\Big( \dfrac{1}{|P|}\int_P{\sum_{k=-\log_2{l(P)}}^{\infty}{\big| T_{m_k}\Pi_kf(y)   \big|^q}}dy\Big)^{1/q}\nonumber\\
&\leq&  \Big(\dfrac{1}{|P|}\int_P{\sum_{k=-\log_2{l(P)}}^{\infty}{2^{kdq(1/u-1)}\big\Vert m_k^{\vee}\Pi_kf(y-\cdot)\big\Vert_{L^u}^{q}}}dy \Big)^{1/q}  
\end{eqnarray}  and then we break (\ref{decompr}) into $\mathcal{R}_{P,u}^{in}$ and $\mathcal{R}_{P,u}^{out}$
where
\begin{eqnarray*}
\mathcal{R}_{P,u}^{in}: =\Big(\dfrac{1}{|P|}\int_P{\sum_{k=-\log_2{l(P)}}^{\infty}{2^{kdq(1/u-1)}\big\Vert m_k^{\vee}\Pi_kf(y-\cdot)\big\Vert_{L^u(B(0,2l(P)))}^{q}}}dy \Big)^{1/q},  
\end{eqnarray*}
\begin{eqnarray*}
\mathcal{R}_{P,u}^{out}:=\Big(\dfrac{1}{|P|}\int_P{\sum_{k=-\log_2{l(P)}}^{\infty}{2^{kdq(1/u-1)}\big\Vert m_k^{\vee}\Pi_kf(y-\cdot)\big\Vert_{L^u(B(0,2l(P))^c)}^{q}}}dy \Big)^{1/q}  .
\end{eqnarray*}
Here $B(0,2l(P))$ denotes the ball of radius $2l(P)$, centered at the origin.\\

By using Minkowski's inequality if $u<q$ and switching two integrals if $u=q$ one obtains
\begin{eqnarray*}
\int_P{\big\Vert m_k^{\vee}\Pi_kf(y-\cdot)\big\Vert_{L^u(B(0,2l(P)))}^q}dy\lesssim \Vert m_k^{\vee}\Vert_{L^u}^q \int_{\widetilde{P}}{|\Pi_kf(y)|^q}dy
\end{eqnarray*} where $\widetilde{P}$ is a dilation of $P$.
Then this proves that 
\begin{eqnarray}\label{dsdsds}
\mathcal{R}_{P,u}^{in}\lesssim \mathcal{K}_u^{0,u}[m]\Big(\dfrac{1}{|P|}\int_{\widetilde{P}}{\sum_{k=-\log_2{l(P)}}^{\infty}{\big|\Pi_kf(y) \big|^q}}dy \Big)^{1/q}
\end{eqnarray} and thus
\begin{eqnarray}\label{dgdgdg}
\big\Vert \sup_{x\in P\in\mathcal{D}}\mathcal{R}_{P,u}^{in}\big\Vert_{L^p(x)}\lesssim \mathcal{K}_u^{0,u}[m]\Big\Vert \mathcal{M}\Big(\sum_{k\in\mathbb{Z}}{|\Pi_kf|^q} \Big)\Big\Vert_{L^{p/q}}^q\lesssim \mathcal{K}_u^{0,u}[m]\Vert f\Vert_{\dot{F}_p^{0,q}}
\end{eqnarray} by the $L^{p/q}$ boundedness of $\mathcal{M}$.\\

To deal with the term corresponding to $\mathcal{R}_{P,u}^{out}$ we choose $\sigma>0$ so that
$d/p<\sigma<s$. Our claim is that for $y\in P$
\begin{eqnarray}\label{excbound}
&&\big\Vert m_k^{\vee}\Pi_kf(y-\cdot) \big\Vert_{L^u(B(0,2l(P))^c)}\nonumber\\
&\lesssim& 2^{-kd(1/u-1)}\big(2^kl(P) \big)^{-(s-\sigma)}\mathcal{K}_u^{s,u}[m]\mathfrak{M}_{\sigma,2^k}\Pi_kf(y).
\end{eqnarray} 
Then one has
\begin{eqnarray}\label{pzpzpz}
\mathcal{R}_{P,u}^{out}&\lesssim& \mathcal{K}_u^{s,u}[m]{\Big(\dfrac{1}{|P|}\int_P{\sum_{k=-\log_2{l(P)}}^{\infty}{(2^kl(P))^{-q(s-\sigma)}\big(\mathfrak{M}_{\sigma,2^k}\Pi_kf(y) \big)^q}}dy \Big)^{1/q}}\nonumber\\
&\lesssim&\mathcal{K}_{u}^{s,u}[m]\Big( \dfrac{1}{|P|}\int_P{\Big( \sup_{k\in\mathbb{Z}}{\mathfrak{M}_{\sigma,2^k}\Pi_kf(y)}\Big)^q}dy\Big)^{1/q}.
\end{eqnarray}
Finally, from the $L^{p}$ boundedness of $\mathcal{M}_q$ and (\ref{max}) with $q=\infty$
it follows 
\begin{eqnarray}\label{klklkl}
\big\Vert \sup_{P:x\in P\in\mathcal{D}}{\mathcal{R}_{P,u}^{out} }\big\Vert_{L^p(x)}&\lesssim& \mathcal{K}_{u}^{s,u}[m]\big\Vert \mathcal{M}_q\big( \sup_{k\in\mathbb{Z}}{\mathfrak{M}_{\sigma,2^k}\Pi_kf}\big) \big\Vert_{L^p}\nonumber\\
&\lesssim& \mathcal{K}_u^{s,u}[m]\big\Vert \sup_{k\in\mathbb{Z}}{\mathfrak{M}_{\sigma,2^k}\Pi_kf}\big\Vert_{L^p}\lesssim\mathcal{K}_u^{s,u}[m]\Vert f\Vert_{\dot{F}_p^{0,\infty}}
\end{eqnarray} 
 and then the embedding $\dot{F}_p^{0,q}\hookrightarrow \dot{F}_p^{0,\infty}$ finishes the proof.
 
To verify (\ref{excbound}) we first observe that
\begin{eqnarray}\label{tftf}
&&\big\Vert m_k^{\vee}\Pi_kf(y-\cdot) \big\Vert_{L^u(B(0,2l(P))^c)}\nonumber\\
&\leq&\Big( \sum_{l=k+\log_2{l(P)}}^{\infty}{\int_{B(0,2l(P))^c}{\big|m_k^{\vee}(z)\chi_{B_l}(2^kz)\big|^u\big|\Pi_kf(y-z)\big|^u}dz}\Big)^{1/u}.
\end{eqnarray}
In fact, the range of $l$ in the sum is $l\geq0$, but  due to the support of $\chi_{B_l}(2^k\cdot)$ and $\chi_{B(0,2l(P))^c}$, if $y\in P$ then the summand vanishes unless $l\geq k+\log_2{l(P)}$.
Therefore, the range $ l\geq 0$ can be replaced by $l\geq k+\log_2{l(P)}$ in the sum.
Moreover, one has
\begin{eqnarray*}
(\ref{tftf})&\leq& \Big(\sum_{l=k+\log_2{l(P)}}^{\infty}{\big| m_k^{\vee}\chi_{B_l}(2^k\cdot)\big|^u\ast\big|\Pi_kf \big|^u(y)} \Big)^{1/u}\\
&\lesssim&\mathfrak{M}_{\sigma,2^k}\Pi_kf(y)\Big(\sum_{l=k+\log_2{l(P)}}^{\infty}{2^{l\sigma u}\big\Vert m_k^{\vee}\chi_{B_l}(2^k\cdot)\big\Vert_{L^u}^u} \Big)^{1/u}.
\end{eqnarray*}
Then (\ref{excbound}) follows by an elementary computation with $s>\sigma$.

\subsubsection{The case $p=\infty$ and $0<q<1$ }

Assume $0<u\leq q<1$, $s>0$, and $\mathcal{K}_u^{s,u}[m]<\infty$.
Let $\mathcal{R}_{P,u}^{in}$ and $\mathcal{R}_{P,u}^{out}$ be defined as before. Then one has
\begin{eqnarray*}
\big\Vert T_mf\big\Vert_{\dot{F}_{\infty}^{0,q}}=\sup_{P\in\mathcal{D}}{\mathcal{R}_{P,u}^{in}}+\sup_{P\in\mathcal{D}}{\mathcal{R}_{P,u}^{out}}.
\end{eqnarray*} 
Moreover the arguments in (\ref{dsdsds}) and (\ref{pzpzpz}) can be extended to $p=\infty$ resulting in
\begin{eqnarray}\label{xbxbxb}
\sup_{P\in\mathcal{D}}{\mathcal{R}_{P,u}^{in}}\lesssim \mathcal{K}_u^{0,u}[m]\Vert f\Vert_{\dot{F}_{\infty}^{0,q}}
\end{eqnarray}
\begin{eqnarray}\label{xvxvxv}
\sup_{P\in\mathcal{D}}{\mathcal{R}_{P,u}^{out}}\lesssim  \mathcal{K}_u^{s,u}[m]\big\Vert \big\{\mathfrak{M}_{\sigma,2^k}\Pi_kf \big\}\big\Vert_{L^{\infty}(l^{\infty})}\lesssim \mathcal{K}_u^{s,u}[m]\Vert f\Vert_{\dot{F}_{\infty}^{0,\infty}}.
\end{eqnarray}
The proof ends by using $\dot{F}_{\infty}^{0,q}\hookrightarrow \dot{F}_{\infty}^{0,\infty}$.

\begin{remark}
Note that (\ref{dgdgdg}) and (\ref{klklkl}) also hold for $0<q<p<\infty$ and $0<u\leq \min{(1,q)}$.
Similarly (\ref{xbxbxb}) and (\ref{xvxvxv}) remain still valid for $0<q<p=\infty$ and $0<u\leq \min{(1,q)}$.

\end{remark}

\subsection{Proof of Theorem \ref{multipliertheorem1} (1)}
Assume $r=\min{(p,q)}<1$, $r<u\leq \infty$, $s=d/r-d/u$, and $t=r$.
Suppose $\mathcal{K}_u^{d/r-d/u,r}[m]<\infty$.
Furthermore, due to proper embeddings in $K$-spaces one may assume $u< 1$.

\subsubsection{The case $0<p<1$ and $p<q\leq \infty$ }
Suppose $0<p<u< 1$, $t=p$, and $s=d/p-d/u$.
As in the proof of Theorem \ref{multipliertheorem0}  let $Q_0$ be any dyadic cubes with side length $2^{-\mu}$ and  let $\{a_Q\}_{Q\in\mathcal{D}}$ be  $\infty$-atoms for $\dot{f}_p^{0,q}$ associated with $Q_0$. Let  $A_{Q_0,k}$ be defined as (\ref{rq}).
Then it suffices to show
\begin{eqnarray*}
\Big\Vert \Big( \sum_{k=\mu}^{\infty}{\big| T_{m_k}( A_{Q_0,k}) \big|^q}\Big)^{1/q}\Big\Vert_{L^p} \lesssim \mathcal{K}_{u}^{d/p-d/u,p}[m]\quad \text{uniformly in $Q_0$}.
\end{eqnarray*} 
As before we decompose the left hand side of the inequality into
\begin{eqnarray*}
\mathcal{G}_{Q_0,p,q}^{1}+\Big(\sum_{j=2}^{\infty}{(\mathcal{G}_{Q_0,p,q}^{j})^p} \Big)^{1/p}
\end{eqnarray*}
and the same argument in (\ref{dgdgdg}) yields that $\mathcal{G}_{Q_0,p,q}^{1}\lesssim \mathcal{K}_{p}^{0,p}[m]\lesssim \mathcal{K}_{u}^{d/p-d/u,p}[m]$.
By using (\ref{dcp1}) and (\ref{dcp2}) with $0<u\leq 1$ one has
\begin{eqnarray}\label{lalala}
\Big(\sum_{j=2}^{\infty}{(\mathcal{G}_{Q_0,p,q}^{j})^p} \Big)^{1/p}&\lesssim&\Big[\sum_{j=2}^{\infty}{\int_{D_j}{\Big(\sum_{k=\mu}^{\infty}{2^{kdq(1/u-1)}\big( \mathcal{L}_{u,j}^{low}[A_{Q_0,k}](x)\big)^q} \Big)^{p/q}}dx} \Big]^{1/p}\nonumber\\
&&+\Big[\sum_{j=2}^{\infty}{\int_{D_j}{\Big(\sum_{k=\mu}^{\infty}{2^{kdq(1/u-1)}\big( \mathcal{L}_{u,j}^{mid}[A_{Q_0,k}](x)\big)^q} \Big)^{p/q}}dx} \Big]^{1/p}\\
&&+\Big[\sum_{j=2}^{\infty}{\int_{D_j}{\Big(\sum_{k=\mu}^{\infty}{2^{kdq(1/u-1)}\big( \mathcal{L}_{u,j}^{high}[A_{Q_0,k}](x)\big)^q} \Big)^{p/q}}dx} \Big]^{1/p}\nonumber
\end{eqnarray}
and (\ref{jkjk}) shows that the first one and the third one are controlled by $\mathcal{K}_u^{0,u}[m]\lesssim \mathcal{K}_{u}^{d/p-d/u,p}[m]$.

Now by using $l^p\hookrightarrow l^q$
(\ref{lalala}) is less than
\begin{eqnarray*}
\Big[\sum_{k=\mu}^{\infty}2^{kdp(1/u-1)}\sum_{j=2}^{\infty}\int_{D_j}{\Big(\sum_{l:2^l\approx 2^{j+k-\mu}}{\int{\big|m_k^{\vee}(y)\chi_{B_l}(2^ky) \big|^u\big|A_{Q_0,k}(x-y)\big|^u}dy} \Big)^{p/u}}dx \Big]^{1/p}
\end{eqnarray*}
and by H\"older's inequality with $u/p>1$ and embedding $l^p\hookrightarrow l^u$ the integral over $D_j$ is dominated by
\begin{eqnarray*}
&&|D_j|^{1-p/u}\Big(\sum_{l:2^l\approx  2^{j+k-\mu} }{\big\Vert m_k^{\vee}\chi_{B_l}(2^k\cdot)\big\Vert_{L^u}^u\big\Vert A_{Q_0,k}\big\Vert_{L^u}^u} \Big)^{p/u}\\
&\lesssim&2^{(j-\mu)d(1-p/u)}\big\Vert A_{Q_0,k}\big\Vert_{L^u}^p2^{-kdp(1/u-1)}\sum_{l:2^l\approx  2^{j+k-\mu} }{\big\Vert (m(2^k\cdot)\varphi)^{\vee}\big\Vert_{L^u(B_l)}^p} \\
&\lesssim&2^{-kdp(1/u-1)}2^{-(k-\mu)dp(1/p-1/u)}\sum_{l:2^l\approx 2^{j+k-\mu}}{\big\Vert \big(m(2^k\cdot)\varphi \big)^{\vee}\big\Vert_{L^u(B_l)}^p}
\end{eqnarray*}
where one has to use $\Vert A_{Q_0,k}\Vert_{L^u}\lesssim 2^{\mu d(1/p-1/u)}$ (which is  due to embedding $l^u\hookrightarrow l^1$ and (\ref{observe})).
Finally, one obtains
\begin{eqnarray*}
(\ref{lalala})\lesssim \mathcal{K}_u^{d/p-d/u,p}[m]\Big(\sum_{k=\mu}^{\infty}{2^{-(k-\mu)d(1/p-1/u)}} \Big)^{1/p}\lesssim \mathcal{K}_u^{d/p-d/u,p}[m].
\end{eqnarray*}

\subsubsection{The case $0<q<1$ and $q<p< \infty$ }
Assume $0<q<u<1$, $t=q$, and $s=d/q-d/u$.
As in the proof of Theorem \ref{multipliertheorem0} it is enough to prove  the case $\frac{1}{1/q-1/u}<p<\infty$ because of real interpolation method (\ref{realinter}) with the case $p=q$.
 Suppose $1/p<{1/q-1/u}$.
Due to (\ref{q<p}) one needs to prove
\begin{eqnarray*}
\big\Vert  \mathcal{N}_q^{\sharp}(\{T_{m_k}({\Pi_k}f)\})\big\Vert_{L^p}\lesssim \mathcal{K}_{u}^{d/q-d/u,q}[m]\Vert f\Vert_{\dot{F}_p^{0,q}}.
\end{eqnarray*}
For $P\in\mathcal{D}$ let ${P}^*$ be the union of $P$ and all dyadic cubes of side length $l(P)$ whose boundaries have non-empty intersection with the boundary of $P$ and let $P^{**}:=(P^*)^*$.
Then one has
\begin{eqnarray*}
\Big(\dfrac{1}{|P|}\int_P{\sum_{k=-\log_2{l(P)}}^{\infty}{\big|T_{m_k}\Pi_kf(y) \big|^q}}dy \Big)^{1/q}\lesssim \mathfrak{R}_P^{in}+\mathfrak{R}_P^{out}
\end{eqnarray*}
where $$\mathfrak{R}_P^{in}:=\Big(\dfrac{1}{|P|}\int_P{\sum_{k=-\log_2{l(P)}}^{\infty}{\big|T_{m_k}\big(\chi_{(P^{**})}\Pi_kf\big)(y) \big|^q}}dy \Big)^{1/q}$$
$$\mathfrak{R}_P^{out}:=\Big(\dfrac{1}{|P|}\int_P{\sum_{k=-\log_2{l(P)}}^{\infty}{\big|T_{m_k}\big(\chi_{(P^{**})^c}\Pi_kf\big)(y) \big|^q}}dy \Big)^{1/q}.$$

From Lemma \ref{less} it follows that
\begin{eqnarray*}
\mathfrak{R}_{P}^{in}              &\lesssim&\mathcal{K}_{q}^{0,q}[m]\Big(\dfrac{1}{|P|}\sum_{k=-\log_2{l(P)}}^{\infty}{\big\Vert \Phi_k \ast \big( \chi_{({P}^{**})}{\Pi_k}f\big)\big\Vert_{L^q}^q}\Big)^{1/q}
\end{eqnarray*}  for some $\Phi_k\in S$ satisfyig $\widehat{\Phi_k}(\xi)=1$ on $2^{k-2}\leq |\xi|\leq 2^{k+2}$.
It should be observed that
\begin{eqnarray*}
\big\Vert \Phi_k\ast \big( \chi_{({P}^{**})}{\Pi_k}f\big)\big\Vert_{L^q}^q&\leq&\sum_{Q\in\mathcal{D}_k,Q\subset {P}^{**}}{\big\Vert \Phi_k\ast\big(\chi_Q{\Pi_k}f\big)\big\Vert_{L^q}^q}
\end{eqnarray*}
and for sufficiently large $M>d(1-q)/q$, by H\"older's inequality with $1/q>1$
\begin{eqnarray*}
\big\Vert \Phi_k\ast \big(\chi_Q{\Pi_k}f\big)\big\Vert_{L^q}^q&\leq&\Vert {\Pi_k}f\Vert_{L^{\infty}(Q)}^q\int_{\mathbb{R}^d}{\Big( \int_{Q}{\big| \Phi_k(x-y)\big|}dy\Big)^q}dx\\
&\lesssim_M&\Vert {\Pi_k}f\Vert_{L^{\infty}(Q)}^q 2^{-kd(1-q)}\Big( \int_{\mathbb{R}^d}{\int_{Q}{\big( 1+2^k|x-c_Q|\big)^M|\Phi_k(x-y)|}dy}dx\Big)^q\\
&\lesssim&2^{-kd}\Vert {\Pi_k}f\Vert_{L^{\infty}(Q)}^q
\end{eqnarray*} where $c_Q$ denotes the center of $Q$. In the last inequality one should use the fact that $1+2^k|x-c_Q|\lesssim 1+2^k|x-y|$ for $y\in Q$.
Therefore
\begin{eqnarray*}
\big\Vert\Phi_k\ast \big( \chi_{({P}^{**})}{\Pi_k}f\big)\big\Vert_{L^q}^q\lesssim\sum_{Q\in\mathcal{D}_k,Q\subset {P}^{**}}{2^{-kd}\Vert {\Pi_k}f\Vert_{L^{\infty}(Q)}^q}
\end{eqnarray*}
and for any $\sigma>0$ this is less than a constant depending on $\sigma$ times
\begin{eqnarray*}
\int_{{P}^{**}}{\big( \mathfrak{M}_{\sigma,2^k}({\Pi_k}f)(x)\big)^q}dx,
\end{eqnarray*} 
by using  (\ref{inftykey}).
This yields
\begin{eqnarray}\label{applyinfty1}
\mathfrak{R}_P^{in}&\lesssim&\mathcal{K}_{q}^{0,q}[m]\Big( \dfrac{1}{|P|}\int_{{P}^{**}}{\sum_{k=-\log_2{l(P)}}^{\infty}{\big(\mathfrak{M}_{\sigma,2^k}{({\Pi_k}f)}(w) \big)^q}}dw\Big)^{1/q}.
\end{eqnarray}
Now for $x\in\mathbb{R}^d$ we take the supremum over $P$ containing $x$ to obtain
$$\sup_{P:x\in P\in\mathcal{D}}\mathfrak{R}_{P}^{in}\lesssim \mathcal{K}_{q}^{0,q}[m] \Big[\mathcal{M}\Big(\sum_{k\in\mathbb{Z}}{\big( \mathfrak{M}_{\sigma,2^k}({\Pi_k}f)\big)^q} \Big)(x)\Big]^{1/q},$$
and by choosing $\sigma>d/q$ and using the $L^{p/q}$ boundedness of $\mathcal{M}$ and (\ref{max})   one proves
\begin{eqnarray*}
\big\Vert \sup_{P:x\in P\in\mathcal{D}}{\mathfrak{R}_{P}^{in} }\big\Vert_{L^p(x)}\lesssim \mathcal{K}_{q}^{0,q}[m]\big\Vert f \big\Vert_{\dot{F}_p^{0,q}}\lesssim \mathcal{K}_{u}^{d/q-d/u,q}[m]\big\Vert f \big\Vert_{\dot{F}_p^{0,q}}.
\end{eqnarray*}

Now it remains to show 
\begin{eqnarray}\label{remaingoal}
\big\Vert \sup_{P:x\in P\in\mathcal{D}}{\mathfrak{R}_{P}^{out} }\big\Vert_{L^p(x)}\lesssim \mathcal{K}_u^{d/q-d/u,q}[m]\Vert f\Vert_{\dot{F}_p^{0,q}}
\end{eqnarray} and we write 
\begin{eqnarray*}
\mathfrak{R}_{P}^{out}=\Big(\dfrac{1}{|P|}\int_P{\sum_{k=-\log_2{l(P)}}^{\infty}{\big|T_{m_k}\big(\Phi_k\ast \big( \chi_{(P^{**})^c}\Pi_kf\big) \big)(y) \big|^q}}dy \Big)^{1/q}
\end{eqnarray*} where $\Phi_k\in S$ whose Fourier transform is $1$ on $2^{k-2}\leq |\xi|\leq 2^{k+2}$.
By using Nikolskii's inequality with $0<u<1$ one has
\begin{eqnarray*}
\big|T_{m_k}\big(\Phi_k\ast \big( \chi_{(P^{**})^c}\Pi_kf\big) \big)(y) \big|&\lesssim& 2^{kd(1/u-1)}\big\Vert m_k^{\vee}(y-\cdot)\Phi_k\ast\big( \chi_{(P^{**})^c}\Pi_kf\big)\big\Vert_{L^u}.
\end{eqnarray*}
Let $$[\mathfrak{R}_P^{out}]_1:=\Big(\dfrac{1}{|P|}\int_P{\sum_{k=-\log_2{l(P)}}^{\infty}{2^{kdq(1/u-1)}\big\Vert m_k^{\vee}(y-\cdot)\Phi_k\ast\big( \chi_{(P^{**})^c}\Pi_kf\big)\big\Vert_{L^u(P^*)}^q   }}dy \Big)^{1/q}$$
$$[\mathfrak{R}_P^{out}]_2:=\Big(\dfrac{1}{|P|}\int_P{\sum_{k=-\log_2{l(P)}}^{\infty}{2^{kdq(1/u-1)}\big\Vert m_k^{\vee}(y-\cdot)\Phi_k\ast\big( \chi_{(P^{**})^c}\Pi_kf\big)\big\Vert_{L^u((P^*)^c)}^q   }}dy \Big)^{1/q}.$$

To estimate $[\mathfrak{R}_P^{out}]_1$ we claim that for $\sigma>0$
\begin{eqnarray*}
\big\Vert m_k^{\vee}(y-\cdot)\Phi_k\ast\big( \chi_{(P^{**})^c}\Pi_kf\big)\big\Vert_{L^u(P^*)}&\lesssim_{\sigma}& \mathcal{K}_{u}^{0,u}[m]2^{-kd(1/u-1)}\mathfrak{M}_{\sigma,2^k}\Pi_kf(y).
\end{eqnarray*}
Indeed, the left hand side is less than
\begin{eqnarray*}
&&\Big(\int_{P^*}{|m_k^{\vee}(y-z)|^u\Big(\int_{(P^{**})^c}{|\Phi_k(z-v)||\Pi_kf(v)|}dv \Big)^u}dz \Big)^{1/u}\\
&\leq&\mathfrak{M}_{\sigma,2^k}\Pi_kf(y)\Big(\int_{P^*}{|m_k^{\vee}(y-z)|^u\Big(\int_{(P^{**})^c}{\big(1+2^k|y-v| \big)^{\sigma}|\Phi_k(z-v)|}dv \Big)^{u}}dz \Big)^{1/u}\\
&\lesssim_{\sigma}&\mathcal{K}_u^{0,u}[m] 2^{-kd(1/u-1)}\mathfrak{M}_{\sigma,2^k}\Pi_kf(y).
\end{eqnarray*} 
Notice that in the last inequality we have used the fact that
$1+2^k|y-v|\lesssim 1+2^k|z-v|$ for $y\in P$, $z\in P^*$, $v\in (P^{**})^c$.

Then this yields that
\begin{eqnarray}\label{applyinfty2}
[\mathfrak{R}_P^{out}]_1\lesssim_{\sigma} \mathcal{K}_{u}^{0,u}[m]\Big(\dfrac{1}{|P|}\int_P{\sum_{k=-\log_2{l(P)}}^{\infty}{\big(\mathfrak{M}_{\sigma,2^k}\Pi_kf(y)\big)^q}}dy \Big)^{1/q}.
\end{eqnarray}
Choosing $\sigma>d/q$ and using $L^{p/q}$ boundedness of $\mathcal{M}$ and (\ref{max}) 
one obtains
\begin{eqnarray*}
\big\Vert \sup_{P:x\in P\in\mathcal{D}}[\mathfrak{R}_P^{out}]_1 \big\Vert_{L^p(x)}\lesssim \mathcal{K}_{u}^{0,u}[m]\Vert f\Vert_{\dot{F}_p^{0,q}}.
\end{eqnarray*}

Now we prove
\begin{eqnarray}\label{remaingoal2}
\big\Vert \sup_{P:x\in P\in\mathcal{D}}[\mathfrak{R}_P^{out}]_2 \big\Vert_{L^p(x)}\lesssim \mathcal{K}_{u}^{d/q-d/u,u}[m]\Vert f\Vert_{\dot{F}_p^{0,\infty}}.
\end{eqnarray}
It is seen that
\begin{eqnarray*}
&&\big\Vert m_k^{\vee}(y-\cdot)\Phi_k\ast\big( \chi_{(P^{**})^c}\Pi_kf\big)\big\Vert_{L^u((P^*)^c)}\\
&\lesssim& \mathfrak{M}_{\sigma,2^k}\big(\Phi_k\ast\big( \chi_{(P^{**})^c}\Pi_kf\big) \big)(y)\Big(\int_{(P^*)^c}{\big|m_k^{\vee}(y-z)\big|^u\big(1+2^k|y-z| \big)^{\sigma u}}dz \Big)^{1/u}.
\end{eqnarray*}
Choosing $d/p<\sigma<d/q-d/u$ and applying the same arguments in (\ref{tftf}) for $y\in P$ the last expression is less than
\begin{eqnarray*}
\mathcal{K}_u^{d/q-d/u,u}[m]\big(2^{k}l(P)\big)^{-(d/q-d/u-\sigma)}\mathfrak{M}_{\sigma,2^k}\big(\Phi_k\ast\big( \chi_{(P^{**})^c}\Pi_kf\big) \big)(y).
\end{eqnarray*}
Observe that
\begin{eqnarray*}
\mathfrak{M}_{\sigma,2^k}\big(\Phi_k\ast\big( \chi_{(P^{**})^c}\Pi_kf\big) \big)\lesssim \mathcal{M}_{d/\sigma}\big( \Phi_k\ast( \chi_{(P^{**})^c}\Pi_kf)\big)\lesssim\mathcal{M}_{d/\sigma}\mathfrak{M}_{\delta,2^k}\Pi_kf.
\end{eqnarray*}

Combining these, $[\mathfrak{R}_P^{out}]_2$ is dominated by
\begin{eqnarray*}
&& \mathcal{K}_u^{d/q-d/u,u}[m]\Big(\dfrac{1}{|P|}\int_P\sum_{k=-\log_2{l(P)}}^{\infty}{ \big(2^{k}l(P)\big)^{-q(d/q-d/u-\sigma)}  \big( \mathcal{M}_{d/\sigma}\mathfrak{M}_{\delta,2^k}\Pi_kf(y)\big)^q } dy\Big)^{1/q}\\
&\lesssim& \mathcal{K}_u^{d/q-d/u,u}[m]\Big(\dfrac{1}{|P|}\int_P\sup_{k\in\mathbb{Z}}{  \big( \mathcal{M}_{d/\sigma}\mathfrak{M}_{\delta,2^k}\Pi_kf (y)\big)^q } dy\Big)^{1/q}
\end{eqnarray*}
and consequently
\begin{eqnarray*}
\big\Vert \sup_{P:x\in P\in\mathcal{D}}{[\mathfrak{R}_P^{out}]_2}\big\Vert_{L^p(x)}&\lesssim& \mathcal{K}_u^{d/q-d/u,u}[m]\big\Vert \mathcal{M}_q\big( \sup_{k\in\mathbb{Z}}{\mathcal{M}_{d/\sigma}{\mathfrak{M}_{\delta,2^k}\Pi_kf}}\big)\big\Vert_{L^p}\\
&\lesssim&\mathcal{K}_u^{d/q-d/u,u}[m]\Vert f\Vert_{\dot{F}_{p}^{0,\infty}}
\end{eqnarray*}
where one has to use the boundedness of $\mathcal{M}_q$ in $L^p$ and (\ref{hlmax}) with $d/\sigma <p$ and $\delta>d/p$.
This proves (\ref{remaingoal2}) and completes the proof of (\ref{remaingoal}).

\subsubsection{The case $p=\infty$ and $0<q<1$ }
Assume $0<q<u<1$, $t=q$, and $s=d/q-d/u$.
Most parts of arguments remain valid and unchanged in this case and will not be repeated.
By using the above arguments,
one obtains
\begin{eqnarray*}
\big\Vert T_mf\big\Vert_{\dot{F}_{\infty}^{0,q}}\lesssim \sup_{P\in\mathcal{D}}{\mathfrak{R}_P^{in}}+\sup_{P\in\mathcal{D}}{[\mathfrak{R}_P^{out}]_1}+\sup_{P\in\mathcal{D}}{[\mathfrak{R}_P^{out}]_2}.
\end{eqnarray*}
Then by (\ref{applyinfty1}) and Lemma \ref{maximal2}
\begin{eqnarray*}
\sup_{P\in\mathcal{D}}{\mathfrak{R}_P^{in}}\lesssim \mathcal{K}_q^{0,q}[m]\Vert f\Vert_{\dot{F}_{\infty}^{0,q}}.
\end{eqnarray*}
By (\ref{applyinfty2}) and Lemma \ref{maximal2}
\begin{eqnarray*}
\sup_{P\in\mathcal{D}}{      [\mathfrak{R}_P^{out}]_1}\lesssim \mathcal{K}_u^{0,u}[m]\Vert f\Vert_{\dot{F}_{\infty}^{0,q}}.
\end{eqnarray*}
Moreover, (\ref{remaingoal2})  still holds even for $p=\infty$. Thus, one has
\begin{eqnarray*}
\sup_{P\in\mathcal{D}}{[\mathfrak{R}_P^{out}]_2}\lesssim \mathcal{K}_u^{d/q-d/u,u}[m]\Vert f\Vert_{\dot{F}_{\infty}^{0,\infty}}.
\end{eqnarray*}
Now we apply $\dot{F}_{\infty}^{0,q}\hookrightarrow \dot{F}_{\infty}^{0,\infty}$ to complete the proof.

\subsection{Proof of Theorem \ref{multipliertheorem2}}
Suppose $1\leq p,q\leq \infty$, $0<u\leq 1$. 
This can be proved by repeating the arguments in the proof of Theorem \ref{multipliertheorem0}. Thus we just give a short description of the proof. As before one may assume $t=u$.\\

For the case $1\leq q<p\leq \infty$ it follows from (\ref{q<p}) and (\ref{decompr}) that
\begin{eqnarray*}
\Vert T_mf\Vert_{\dot{F}_p^{0,q}}\lesssim \big\Vert \sup_{P:x\in P\in\mathcal{D}}{\mathcal{R}^{in}_{P,u}}\big\Vert_{L^p(x)}+\big\Vert \sup_{P:x\in P\in\mathcal{D}}{\mathcal{R}_{P,u}^{out}}\big\Vert_{L^p(x)}.
\end{eqnarray*}
Then, as mentioned in the remark after the proof of Theorem \ref{multipliertheorem0}, one can apply
 (\ref{dgdgdg}), (\ref{klklkl}) for $p<\infty$ and  (\ref{xbxbxb}), (\ref{xvxvxv}) for $p=\infty$.\\

The case $1<p<q\leq \infty$ follows via daulity.\\

When $p=1<q\leq \infty$, it can be proved by using $\varphi$-transform and $\infty$-atoms for $\dot{f}_1^{0,q}$.

\subsection{Proof of Theorem \ref{multipliertheorem3}(1)}
Let $1\leq p,q\leq \infty$, $p\not= q$, and $1<u\leq \infty$. One may assume $t=1$
\subsubsection{The case $1< p<\infty$, $1<q\leq \infty$}
Suppose that $s=d-d/u$.
Our claim is  the pointwise estimate
\begin{eqnarray}\label{pointestmm}
\big| T_{m_k}({\Pi}_kf)(x)\big|\lesssim \mathcal{K}_u^{d-d/u,1}[m]\mathcal{M}_{u'}\big( \Pi_kf\big)(x)
\end{eqnarray} where $1/u+1/u'=1$.
Then for $u'<p<\infty$ and $u'<q\leq\infty$ the $\dot{F}_p^{0,q}$ boundedness of $\mathcal{M}_{u'}$ gives the boundedness result of $T_m$ and  duality argument (\ref{dual}) can be applied for $1<p,q<u$. Finally the usage of the complex interpolation (\ref{complexinter}) finishes the proof by giving the condition $\big|1/p-1/q\big|<1-1/u$.

For the estimation of (\ref{pointestmm})  the left hand side can be bounded by 
\begin{eqnarray*}
\sum_{l=0}^{\infty}\big| \big( m_k^{\vee}\chi_{B_l}(2^k\cdot)\big)\ast \big(\Pi_kf\big)(x)\big|.
\end{eqnarray*} Then using H\"older's inequality with $u>1$ each summand is dominated by $$2^{ld(1-1/u)}\big\Vert \big( m(2^k\cdot)\varphi\big)^{\vee}\big\Vert_{L^u(B_l)}\mathcal{M}_{u'}(\Pi_kf)(x)$$
and (\ref{pointestmm}) follows consequently.

\subsubsection{The case $q=1< p<u$}
One needs to prove
\begin{eqnarray}\label{ffgoal}
\Big(\int{\Big( \sum_{k\in\mathbb{Z}}{\big|T_{m_k}({\Pi}_kf)(x) \big|}\Big)^{p}}dx \Big)^{1/p}\lesssim \mathcal{K}_{u}^{d-d/u,1}[m]\Vert f\Vert_{F_p^{0,1}}.
\end{eqnarray}
Using duality (\ref{dual})  the left hand side of (\ref{ffgoal}) can be written as
\begin{eqnarray*}\label{estest}
&&\sup_{\Vert g\Vert_{L^{p'}}\leq 1}{\int{\sum_{k\in\mathbb{Z}}{\big|T_{m_k}({\Pi}_kf)(x)\big|g(x)}}dx}\nonumber\\
&\leq&\sup_{\Vert g\Vert_{L^{p'}}\leq 1}{\sum_{k\in\mathbb{Z}}{\sum_{l=0}^{\infty}{\int{\big| \big(m_k^{\vee}\chi_{B_l}(2^{k}\cdot)\big)\ast({\Pi}_kf)(x)g(x)\big|}dx}}}\nonumber\\
&\leq&\sup_{\Vert g\Vert_{L^{p'}}\leq 1}{\sum_{k\in\mathbb{Z}}{\sum_{l=0}^{\infty}{\sum_{Q\in\mathcal{D}_{k-l}}{\int{\big|\big(m_k^{\vee}\chi_{B_l}(2^{k}\cdot)\big)\ast\big(\chi_Q{\Pi}_kf\big)(x)  (g\chi_{\widetilde{Q}})(x)\big|}dx}}}}\nonumber
\end{eqnarray*} where $\widetilde{Q}$ is a dilation of $Q$.
In the last inequality $\chi_{\widetilde{Q}}$ appears in the integral because of  the supports of $\chi_{B_l}(2^k\cdot)$ and $\chi_Q$.

Then using H\"older's inequality and Young's inequality the integral in the last expression is less than
\begin{eqnarray*}
&&\big\Vert m_k^{\vee}\chi_{B_l}(2^k\cdot)\big\Vert_{L^u} \big\Vert \chi_Q{\Pi}_kf\big\Vert_{L^1}\big\Vert g\chi_{\widetilde{Q}}\big\Vert_{L^{u'}} \\
&\lesssim& 2^{ld(1-1/u)}\big\Vert \big(m(2^k\cdot)\varphi \big)^{\vee}\big\Vert_{L^u(B_l)}\int_Q{|{\Pi}_kf(x)|\mathcal{M}_{u'}g(x)}dx
\end{eqnarray*}
and thus the supremum is dominated by a constant times
\begin{eqnarray*}
&&\mathcal{K}_{u}^{d-d/u,1}[m]\sup_{\Vert g\Vert_{L^{p'}\leq 1}}{\int_{\mathbb{R}^d}{\sum_{k\in\mathbb{Z}}{|{\Pi}_kf(x)|}\mathcal{M}_{u'}g(x)  }dx}.
\end{eqnarray*}
Finally, (\ref{ffgoal}) follows from H\"older's inequality with $u>1$ and the $L^{p'}$ boundedness of $\mathcal{M}_{u'}$ with $p'>u'$.

\section{\textbf{Proof of Theorem \ref{sharpbesov} $-$ \ref{sharpthm2}}}\label{example}
In what follows let $\eta$, $\widetilde{\eta}$ denote Schwartz functions so that $\eta \geq 0$, $\eta(x)\geq c$ on $\{x:|x|\leq 1/100\}$ for some $c>0$, $Supp(\widehat{\eta})\subset \{\xi: |\xi|\leq 1/1000\}$, $\widehat{\widetilde{\eta}}(\xi)=1$ for $|\xi|\leq 1/1000$, and $Supp(\widehat{\widetilde{\eta}})\subset \{\xi: |\xi|\leq 1/100\}$. Let $e_1:=(1,0,\dots,0)\in\mathbb{R}^d$.
Moreover let $\{\lambda_1,\lambda_2,\dots\}$ be a sequence of lattices in $\mathbb{R}^d$ such that 
\begin{eqnarray}\label{secondition}
|\lambda_n|\leq \sqrt{d}n^{1/d}.
\end{eqnarray} 
One way to select such a sequence is as follows.
For each $k=1,2,\dots$ let $$\lambda_{k^d}:=(k,0,0,\dots,0).$$
It should be observed that there are at most $d(k+1)^{d-1}$ integers between $k^d$ and $(k+1)^{d}$, and
there exist at least $2d(2k-1)^{d-1}$ lattices on the surface of cube $[-k,k]^d$.
Since $d(k+1)^{d-1}\leq 2d(2k-1)^{d-1}$ one can choose lattices $\lambda_{k^d+1},\lambda_{k^d+2},\dots,\lambda_{(k+1)^d-1}$ on the surface of the cube and then clearly the length of those lattices is less than $\sqrt{d}k$, which yields (\ref{secondition}).

\subsection{Proof of Theorem \ref{sharpbesov}(1)}

\subsubsection{The case $0<p\leq 1$}
Suppose $u\leq p<t$ and let
\begin{eqnarray*}
m(\xi):=\sum_{k=10}^{\infty}{\frac{1}{k^{1/p}}e^{2\pi i\langle \xi-e_1,2^{k}e_1\rangle}\widehat{\widetilde{\eta}}(\xi-e_1)}.
\end{eqnarray*}
Then due to the support of $m$ and $\varphi$ one has
\begin{eqnarray*}
\big(m(2^k\cdot)\varphi\big)^{\vee}(x)=\begin{cases}
\sum_{n=10}^{\infty}{\frac{1}{n^{1/p}}\big(2^{-kd}\widetilde{\eta}(x/2^k+2^ne_1)e^{2\pi i\langle \cdot,2^{-k}e_1\rangle} \big)\ast \varphi^{\vee}(x)},&-2\leq k\leq 2\\
0, & otherwise.
\end{cases}
\end{eqnarray*}
Let $-2\leq k\leq 2$. Then for arbitrary $M>0$
\begin{eqnarray*}
\big|\big(m(2^k\cdot)\varphi \big)^{\vee}(x) \big|&\lesssim&\sum_{n=10}^{\infty}{\frac{1}{n^{1/p}}\big|{\widetilde{\eta}}(\cdot/2^k+2^ne_1) \big|\ast |\varphi^{\vee}|(x)}\\
&\lesssim_M& \sum_{n=10}^{\infty}{\frac{1}{n^{1/p}}\frac{1}{\big(1+|x+2^ne_1| \big)^{2M}}}
\end{eqnarray*}
and 
\begin{eqnarray*}
\big\Vert \big( m(2^k\cdot)\varphi\big)^{\vee}\big\Vert_{L^u(B_l)}\lesssim_M \Big(\sum_{n=10}^{\infty}{\frac{1}{n^{u/p}}\big\Vert \big( 1+|\cdot+2^ne_1|\big)^{-{2M}}\big\Vert_{L^u(B_l)}^u} \Big)^{1/u}.
\end{eqnarray*}
We choose $M>1/u$ sufficiently large.
Since 
\begin{eqnarray*}
\big\Vert \big( 1+|\cdot+2^ne_1|\big)^{-{2M}}\big\Vert_{L^u(B_l)}\lesssim \begin{cases}
1,&\\
2^{-lM},& n\leq l-1\\
2^{-nM}, & n\geq l+2
\end{cases}
\end{eqnarray*}
one has
\begin{eqnarray*}
\big\Vert \big(m(2^k\cdot)\varphi \big)^{\vee}\big\Vert_{K_u^{0,t}}&\lesssim& \Big(\sum_{l=0}^{\infty}{\Big(\sum_{n=\max{(10,l+2)}}^{\infty}{\frac{1}{n^{u/p}}2^{-nMu}} \Big)^{t/u}} \Big)^{1/t}\\
&&+\Big(\sum_{l=9}^{\infty}\Big(\sum_{n=\max{(10,l)}}^{l+1}{\frac{1}{n^{u/p}}} \Big)^{t/u} \Big)^{1/t}+\Big(\sum_{l=11}^{\infty}{\Big(\sum_{n=10}^{l-1}{\frac{1}{n^{u/p}}2^{-lMu}} \Big)^{t/u}} \Big)^{1/t}\\
&\lesssim&1+\Big(\sum_{l=9}^{\infty}{\frac{1}{l^{t/p}}} \Big)^{1/t}.
\end{eqnarray*}
Since $t<p$ this converges, and finally one has $\mathcal{K}_u^{0,t}[m]<\infty$.

Now let $f(x):={\eta}(x)e^{2\pi i\langle x,e_1\rangle}$.
It is clear that $\Vert f\Vert_{\dot{B}_p^{0,q}}<\infty$.
On the other hand,
one has
\begin{eqnarray*}
T_mf(x)=e^{2\pi i\langle x,e_1\rangle}\sum_{k=10}^{\infty}{\dfrac{1}{k^{1/p}}\eta(x+2^ke_1)}
\end{eqnarray*} and thus
\begin{eqnarray*}
\big\Vert T_mf\big\Vert_{\dot{B}_p^{0,q}}&\gtrsim& \Big\Vert  \sum_{k=10}^{\infty}{\frac{1}{k^{1/p}}\eta(\cdot+2^ke_1)}    \Big\Vert_{L^p}\\
  &\geq&\Big(\sum_{n=10}^{\infty}{\int_{|x-2^ne_1|\leq 1/100}{\Big( \sum_{k=10}^{\infty}{\frac{1}{k^{1/p}}\eta(x+2^ke_1)}   \Big)^p     }dx} \Big)^{1/p}\\
  &\gtrsim&\Big(\sum_{n=10}^{\infty}{\frac{1}{k}} \Big)^{1/p}=\infty.
\end{eqnarray*}

\subsubsection{The case $1<p\leq \infty$}

Now suppose $1<p\leq \infty$ and $0<u\leq 1<t$ and let
\begin{eqnarray*}
m(\xi):=\sum_{k=10}^{\infty}{\frac{1}{k}e^{2\pi i\langle \xi-e_1,2^{k}e_1\rangle}\widehat{\widetilde{\eta}}(\xi-e_1)}.
\end{eqnarray*}
Then by the same arguments one can prove $\mathcal{K}_u^{0,t}[m]<\infty$.
However, this is unbounded function near $e_1$ and this proves that $T_m$ is not bounded on $L^p$ for $1<p<\infty$. Due to the support of $\widehat{\eta}$ this implies that $T_m$ is not bounded on $\dot{B}_p^{0,q}$ for $1<p<\infty$.

When $p=\infty$, let
\begin{eqnarray*}
f(x)=\sum_{k=10}^{\infty}{{\eta}(x-2^ke_1)e^{2\pi i\langle x,e_1\rangle}}.
\end{eqnarray*}
Then since $\widehat{f}\subset \{\xi:|\xi|\approx 1\}$ one has $\Vert f\Vert_{\dot{B}_{\infty}^{0,q}}\lesssim \Vert f\Vert_{L^{\infty}}\lesssim1$.
Moreover, 
\begin{eqnarray*}
 T_mf(x)=e^{2\pi i\langle x,e_1\rangle}\sum_{k=10}^{\infty}{\sum_{n=10}^{\infty}{\frac{1}{k}\eta(x-2^ne_1+2^ke_1)}}
\end{eqnarray*}
and
\begin{eqnarray*}
\big\Vert T_mf\big\Vert_{\dot{B}_{\infty}^{0,q}}&\gtrsim&\big\Vert  T_mf  \big\Vert_{L^{\infty}}\geq \Big\Vert \eta\sum_{k=10}^{\infty}{\frac{1}{k}}\Big\Vert_{L^{\infty}}=\infty.
\end{eqnarray*}

\subsection{Proof of Theorem \ref{sharpbesov}(2)}

\subsubsection{The case $0<p<1$}
Suppose $0<p<1$, $p<u\leq \infty$,  $p<t$, and $s=d/p-d/u$.
Let $1/t<\tau<1/p$ and   $$h^{(p,\tau)}(x):=\dfrac{1}{(1+|x|^2)^{d/(2p)}}\dfrac{1}{\big( 1+\log{(1+|x|)}\big)^{\tau}}.$$
Let 
\begin{eqnarray*}
K^{(p,\tau)}(x):=h^{(p,\tau)}\ast \eta(x)e^{-2\pi i \langle x,e_1\rangle}
\end{eqnarray*}
 and 
\begin{eqnarray*}
m^{(p,\tau)}(\xi):={\widehat{K^{(p,\tau)}}(\xi)}.
\end{eqnarray*} 
Then it follows that
\begin{eqnarray*}
\big(m(2^k\cdot)\varphi(x) \big)^{\vee}=\begin{cases}
\big(2^{-kd}K^{(p,\tau)}(2^k\cdot) \big)\ast \varphi^{\vee}(x), & -2\leq k\leq 2\\
0, & otherwise.
\end{cases}
\end{eqnarray*} 
By using the fact that $h(x+y)\leq \frac{h(x)}{h(y)}$ one has
\begin{eqnarray*}
\big|\big(m(2^k\cdot)\varphi \big)^{\vee}(x) \big|\lesssim h^{(p,\tau)}(x)
\end{eqnarray*}
and thus
\begin{eqnarray*}
\big\Vert \big( m(2^k\cdot)\varphi\big)^{\vee}\big\Vert_{L^u(B_l)}  &\lesssim&2^{-ld(1/p-1/u)}(1+l)^{-\tau}.
\end{eqnarray*}
This leads to
\begin{eqnarray*}
\sup_{k\in\mathbb{Z}}\big\Vert \big(m^{(p,\tau)}(2^k\cdot)\varphi\big)^{\vee}\big\Vert_{K_u^{d/p-d/u,t}}
&\lesssim& \Big( \sum_{l=0}^{\infty}{(1+l)^{-\tau t}}\Big)^{1/t}<\infty
\end{eqnarray*} because $\tau>1/t$.

Let $f(x):=\widetilde{\eta}(x/2^{10})e^{2\pi i\langle x,e_1\rangle}$. Then clearly $\Vert f\Vert_{\dot{B}_{p}^{0,q}}\approx\Vert f\Vert_{L^p}\lesssim 1$, but
\begin{eqnarray*}
\Vert T_{m^{(p,\tau)}}f\Vert_{\dot{B}_{p}^{0,q}}&\gtrsim&\Vert h^{(p,\tau)}\ast \eta\Vert_{L^p}  \gtrsim\Vert h^{(p,\tau)}\Vert_{L^p}
\end{eqnarray*} which diverges since $\tau < 1/p$.

\subsubsection{The case $1\leq p\leq \infty$}
Suppose $r=1<u< \infty$ and $1<t$. We additionally assume $1<t<u$ because other cases follow by embedding.
 We choose $\delta$ between $1/t$ and $1$ and let
\begin{eqnarray*}
m(\xi):=\sum_{n=10}^{\infty}{\frac{1}{n}\frac{1}{(\log{n})^{\delta}}e^{-2\pi i\langle \lambda_n,\xi -e_1\rangle}\widehat{\widetilde{\eta}}(\xi-e_1)}.
\end{eqnarray*}
Then
\begin{eqnarray*}
\big(m(2^k\cdot)\varphi \big)^{\vee}(x)=\begin{cases}
\displaystyle\sum_{n=10}^{\infty}{\frac{1}{n}\frac{1}{(\log{n})^{\delta}}\big(2^{-kd}\widetilde{\eta}(\cdot/2^k-\lambda_n) \big)\ast \varphi^{\vee}(x)}, & -2\leq k\leq 2\\
0,& otherwise.
\end{cases}
\end{eqnarray*}
Therefore one has that for $M>0$
\begin{eqnarray*}
\big|\big(m(2^k\cdot)\varphi \big)^{\vee}(x) \big|\lesssim_M \sum_{n=10}^{\infty}{\frac{1}{n}\frac{1}{(\log{n})^{\delta}}\frac{1}{\big( 1+|x-\lambda_n|\big)^{M}}}
\end{eqnarray*}
and $\big\Vert \big( m(2^k\cdot)\varphi\big)^{\vee}\big\Vert_{L^u(B_l)}$ is bounded by $\Omega_1+\Omega_2$ 
where
\begin{eqnarray*}
\Omega_1&=&\Big( \int_{B_l}{\Big(\sum_{n:2^l\geq 4\sqrt{d}n^{1/d}}{\frac{1}{n}\frac{1}{(\log{n})^{\delta}}\frac{1}{(1+|x-\lambda_n| )^M}} \Big)^{u}}dx\Big)^{1/u}
\end{eqnarray*}
\begin{eqnarray*}
\Omega_2&=&\Big( \int_{B_l}{\Big(\sum_{n:2^l<4\sqrt{d}n^{1/d}}{\frac{1}{n}\frac{1}{(\log{n})^{\delta}}\frac{1}{(1+|x-\lambda_n| )^M}} \Big)^{u}}dx\Big)^{1/u}.
\end{eqnarray*}

If $2^l\geq 4\sqrt{d}n^{1/d}$ and $x\in B_l$, then $\frac{1}{(1+|x-\lambda_n|)^M}\lesssim 2^{-lM}$ due to  (\ref{secondition}). 
By choosing $M$ sufficiently large
 one obtains
\begin{eqnarray*}
\Big(\sum_{l=0}^{\infty}{2^{ltd(1-1/u)}({\Omega_1})^t} \Big)^{1/t} \lesssim 1.
\end{eqnarray*}

To estimate the term associated with $\Omega_2$ we choose a constant $r$ so that $1/t-1/u<r/u<\delta-1/u$.
Using H\"older's inequalities with $u/t>1$ and $u>1$ it follows that
\begin{eqnarray*}
&&\Big(\sum_{l=0}^{\infty}{2^{ldt(1-1/u)}({\Omega_2})^t} \Big)^{1/t}\lesssim\Big[\sum_{l=0}^{\infty}\Big(\int_{B_l}\Big(    \sum_{n:2^l<4\sqrt{d}n^{1/d}}{\frac{1}{n^{1/u}}\frac{1}{(\log{n})^{\delta}}\frac{1}{(1+|x-\lambda_n|)^M}}           \Big)^{u}dx \Big)^{t/u} \Big]^{1/t}\\
&\leq&\Big( \sum_{l=0}^{\infty}{\frac{1}{(1+l)^{rt/(u-t)}}}\Big)^{1/t-1/u}\Big( \sum_{l=0}^{\infty}{(1+l)^r\int_{B_l}{\Big( \sum_{n:2^l<4\sqrt{d}n^{1/d}}{\frac{1}{n^{1/u}}\frac{1}{(\log{n})^{\delta}}\frac{1}{(1+|x-\lambda_n|)^M}}\Big)^{u}  }dx}\Big)^{1/u}\\
&\lesssim&\Big( \sum_{l=0}^{\infty}{\int_{B_l}{ \Big( \sum_{n:2^l<4\sqrt{d}n^{1/d}}{\frac{1}{n^{1/u}}\frac{1}{(\log{n})^{\delta-r/u}}\frac{1}{(1+|x-\lambda_n|)^M}}\Big)^{u}  }dx}\Big)^{1/u}\\
&\lesssim&\Big(\int_{\mathbb{R}^d}{\sum_{n=10}^{\infty}{\dfrac{1}{n}\dfrac{1}{(\log{n})^{\delta u-r}}\frac{1}{(1+|x-\lambda_n|)^M}} \Big( \sum_{l=10}^{\infty}{\frac{1}{(1+|x-\lambda_l|)^M}}\Big)^{u-1}}dx \Big)^{1/u}.
\end{eqnarray*} 
Choosing $M>d$ the last term is less than a constant times
$$\Big( \sum_{n=10}^{\infty}{\frac{1}{n}\frac{1}{(\log{n})^{\delta u-r}}}\Big)^{1/u}$$
because $\sum_{l=10}^{\infty}{\frac{1}{(1+|x-\lambda_l|)^M}}$ is finite. This proves 
\begin{eqnarray*}
\Big(\sum_{l=0}^{\infty}{2^{ltd(1-1/u)}({\Omega_2})^t} \Big)^{1/t} \lesssim_M 1
\end{eqnarray*} because $\delta u-r>1$.
Finally, one has
$\mathcal{K}_{u}^{d(1-1/u),t}[m]< \infty$.

Now due to the support of $m$ it follows that \begin{eqnarray*}
\Vert T_m\Vert_{\dot{B}_p^{0,q}\to \dot{B}_{p}^{0,q}}&\approx&\Vert T_m\Vert_{L^{p}\to L^{p}}
\end{eqnarray*} and thus one needs to confirm the $L^p$-boundedness of $T_m$.

For $1<p<\infty$ $T_m$ is not bounded on $L^p$ because $\delta<1$ implies $m\not\in L^{\infty}$.

For $p=\infty$ let $$f(x)=\sum_{n=10}^{\infty}{{\eta}(x+\lambda_n)e^{2\pi i\langle x,e_1\rangle}}.$$
Then it is clear that $f\in L^{\infty}$. On the other hand, one has
 $$T_mf(x)=e^{2\pi i\langle x,e_1\rangle}\sum_{k=10}^{\infty}{\sum_{n=10}^{\infty}{\frac{1}{k}\frac{1}{(\log{k})^{\delta}}\eta(x+\lambda_n-\lambda_k)}}$$
and thus $$\Vert T_mf\Vert_{L^{\infty}}\gtrsim \sum_{k=10}^{\infty}{\frac{1}{k}\frac{1}{(\log{k})^{\delta}}}=\infty.$$
The case $p=1$ is immediate by duality $(L^1)^*=L^{\infty}$.

\subsection{{Proof of Theorem \ref{sharpthm0}}}
Suppose $0<p,q\leq\infty$, $p\not= q$, $t>0$, and $0<u\leq \min{(1,p,q)}$.
For $n\in \mathbb{N}$ let $\zeta_n:=10n$.
\subsubsection{Construction of a multiplier}
Let 
\begin{eqnarray}\label{fffexample}
m(\xi):=\sum_{n=10}^{\infty}{\widehat{\widetilde{\eta}}\big((\xi-2^{\zeta_n}e_1)/2^{\zeta_n} \big)e^{2\pi i\langle 2^{\zeta_n}e_1,\xi-2^{\zeta_n}e_1\rangle}}.
\end{eqnarray}
Then one has 
\begin{eqnarray*}
m(2^k\xi)\varphi(\xi)=\begin{cases}
\widehat{\widetilde{\eta}}\Big( \dfrac{2^k\xi-2^{\zeta_m}e_1}{2^{\zeta_m}}\Big)e^{2\pi i\langle 2^{\zeta_m}e_1,2^k\xi-2^{\zeta_m}e_1 \rangle}\varphi(\xi), & \zeta_m-2\leq k\leq \zeta_m+2, m\geq 10\\
0, & otherwise.
\end{cases}
\end{eqnarray*}
Let $k=\zeta_m-j$ for some $j\in \{-2,-1,0,1,2\}$.
Then by elementary computation
\begin{eqnarray*}
\big|\big( m(2^k\cdot)\varphi\big)^{\vee}(x)\big|&\leq& \big|2^{jd}\widetilde{\eta}(2^j\cdot+2^{2\zeta_m}e_1) \big|\ast|\varphi^{\vee}|(x)\\
&\lesssim_M&\dfrac{1}{\big(1+|2^jx+2^{2\zeta_m}e_1| \big)^{M}}\int_{\mathbb{R}^d}{\big(1+2^j|y| \big)^M|\varphi^{\vee}(y)|}dy\\
&\lesssim_M&\dfrac{1}{\big(1+|x+2^{2\zeta_m-j}e_1| \big)^M}
\end{eqnarray*}
and thus
\begin{eqnarray}\label{estfff}
\big\Vert \big(m(2^k\cdot)\varphi \big)^{\vee}\big\Vert_{L^u(B_l)}&\lesssim& \Big(\int_{B_l}{\dfrac{1}{\big(1+|x+2^{2\zeta_m-j}e_1| \big)^{Mu}}}dx \Big)^{1/u}\nonumber\\
&\lesssim_N&\begin{cases}
1 \quad &
\\
2^{-lN} \quad & \quad2^l\geq 2^{2\zeta_m+3}
\\
2^{-\zeta_mN} \quad &
\quad  2^l\leq 2^{2\zeta_m-4}
\end{cases}
\end{eqnarray} for sufficiently large $N$.

Finally, one obtains
\begin{eqnarray*}
\big\Vert \big(m(2^k\cdot)\varphi \big)^{\vee}\big\Vert_{K_u^{0,t}}&=&\Big(\sum_{l=0}^{\infty}{\big\Vert \big(m(2^k\cdot)\varphi \big)^{\vee}\big\Vert_{L^u(B_l)}^t} \Big)^{1/t}\\
&\lesssim&\Big(\sum_{l=0}^{2\zeta_m-4}{2^{-\zeta_mNt}} \Big)^{1/t}+\Big(\sum_{l=2\zeta_m-3}^{2\zeta_m+2}{1} \Big)^{1/t}+\Big(\sum_{l=2\zeta_m+3}^{\infty}{2^{-lNt}} \Big)^{1/t}\lesssim 1
\end{eqnarray*} and this proves $\mathcal{K}_{u}^{0,t}[m]<\infty$.

\subsubsection{The case $0<p<q\leq \infty$}
Let $a_k:=k^{-1/q}(\log{k})^{-\epsilon}$ for $1/q<\epsilon<1/p$.
Let \begin{eqnarray}\label{ffunction}
f(x):=\sum_{n=10}^{\infty}{a_{\zeta_n}\eta(x)e^{2\pi i\langle x,2^{\zeta_n}e_1\rangle}}.
\end{eqnarray}
Then 
\begin{eqnarray*}
\Pi_kf(x)=\begin{cases}
a_{\zeta_m}\big(\eta e^{2\pi i\langle x,2^{\zeta_m}e_1\rangle} \big)\ast\phi_k(x), & \zeta_m-1\leq k\leq \zeta_m+1, m\geq 10\\
0,& otherwise.
\end{cases}
\end{eqnarray*}
For each $\zeta_m-1\leq k\leq \zeta_m+1, m\geq 10$ it is clear that for $M>0$
\begin{eqnarray*}
\big| \Pi_kf(x)\big|&\leq&a_{\zeta_m}|\eta|\ast |\phi_k|(x)\lesssim_M a_{\zeta_m}\frac{1}{(1+|x|)^M}.
\end{eqnarray*}
Choosing $M>d/p$  one has
\begin{eqnarray*}
\Vert f\Vert_{\dot{F}_p^{0,q}}\lesssim \Big(\sum_{m=10}^{\infty}\sum_{k=\zeta_m-1}^{\zeta_m+1}{a_{\zeta_m}^q} \Big)^{1/q}\lesssim \Big(\sum_{k=10}^{\infty}a_k^q \Big)^{1/q}<\infty.
\end{eqnarray*}
On the other hand, one has
\begin{eqnarray*}
T_mf(x)=\sum_{n=10}^{\infty}{a_{\zeta_n}\eta(x+2^{\zeta_n}e_1)e^{2\pi i\langle x,2^{\zeta_n}e_1\rangle}}
\end{eqnarray*}
and
\begin{eqnarray*}
\Vert T_mf\Vert_{\dot{F}_p^{0,q}}&=&\Big\Vert \Big( \sum_{k\in\mathbb{Z}}{\big|\phi_k\ast T_mf\big|^q}\Big)^{1/q}   \Big\Vert_{L^p}\gtrsim \Big\Vert \Big(\sum_{k\in\mathbb{Z}}{\big|(\phi_{k-1}+\phi_k+\phi_{k+1})\ast T_mf\big|^q}\Big)^{1/q} \Big\Vert_{L^p}\\
&\geq&\Big(\int_{\mathbb{R}^d}{\Big( \sum_{n=10}^{\infty}{|a_{\zeta_n}|^q\big|\eta(x+2^{\zeta_n}e_1) \big|^q}\Big)^{p/q}}dx \Big)^{1/p}\\
&\geq&\Big(\sum_{m=10}^{\infty}{\int_{|x+2^{\zeta_m}e_1|\leq 1/100}{|a_{\zeta_m}|^p\big|\eta(x+2^{\zeta_m}e_1) \big|^p}dx} \Big)^{1/p}\\
&\gtrsim&\Big( \sum_{m=10}^{\infty}{|a_{\zeta_m}|^p}\Big)^{1/p}=\infty.
\end{eqnarray*}

\subsubsection{The case $0<q<p< \infty$}
Choosing $b_k:=k^{-1/p}(\log{k})^{-\delta}$ for $1/p<\delta<1/q$  let
\begin{eqnarray}\label{gfunction}
g(x):=\sum_{n=10}^{\infty}{b_{\zeta_n}\eta(x-2^{\zeta_n}e_1)e^{2\pi i\langle x,2^{\zeta_n}e_1\rangle}}.
\end{eqnarray}
Similarly one has
\begin{eqnarray*}
\Pi_kg(x)=\begin{cases}
b_{\zeta_m}\big( \eta(\cdot-2^{\zeta_m}e_1)e^{2\pi i\langle \cdot,2^{\zeta_m}e_1\rangle}\big)\ast\phi_k(x),&\zeta_m-1\leq k\leq \zeta_m+1, m\geq 10\\
0,& otherwise.
\end{cases}
\end{eqnarray*}
This implies that  for each $\zeta_m-1\leq k\leq \zeta_m+1, m\geq 10$ 
\begin{eqnarray*}
|\Pi_kg(x)|\leq b_{\zeta_m}\big|\eta(\cdot-2^{\zeta_m}e_1)\big|\ast |\phi_k|(x)\lesssim_M b_{\zeta_m}\frac{1}{\big(1+|x-2^{\zeta_m}e_1| \big)^M}
\end{eqnarray*}
and thus
\begin{eqnarray*}
\Vert g\Vert_{\dot{F}_p^{0,q}}&\lesssim_M&\Big(\int_{\mathbb{R}^d}{\Big( \sum_{m=10}^{\infty}\sum_{k=\zeta_m-1}^{\zeta_m+1}{|b_{\zeta_m}|^q\frac{1}{\big(1+|x-2^{\zeta_m}e_1| \big)^{Mq}}}\Big)^{p/q}}dx \Big)^{1/p}.
\end{eqnarray*}
Choosing $M>d/q$ and using H\"older's inequality with $p/q>1$ this is less than
\begin{eqnarray*}
\Big(\int_{\mathbb{R}^d}{\sum_{m=10}^{\infty}\sum_{k=\zeta_m-1}^{\zeta_m+1}{|b_{\zeta_m}|^p\frac{1}{\big(1+|x-2^{\zeta_m}e_1| \big)^{Mq}}}}dx \Big)^{1/p}\lesssim \Big(\sum_{k=10}^{\infty}{|b_k|^p} \Big)^{1/p}<\infty.
\end{eqnarray*}

However, 
\begin{eqnarray*}
T_mg(x)=\sum_{m=10}^{\infty}{b_{\zeta_m}e^{2\pi i\langle x,2^{\zeta_m}e_1\rangle}\eta(x)}
\end{eqnarray*} and thus
\begin{eqnarray*}
\Vert T_mg\Vert_{\dot{F}_p^{0,q}}&\gtrsim& \Big\Vert \Big(\sum_{k\in\mathbb{Z}}{\big|\big( \phi_{k-1}+\phi_k+\phi_{k+1}\big)\ast T_mg \big|^q}\Big)^{1/q}\Big\Vert_{L^p}\\
&\geq& \Vert \eta\Vert_{L^p}\Big( \sum_{n=10}^{\infty}{|b_{\zeta_n}|^q}\Big)^{1/q}=\infty.
\end{eqnarray*}

\subsubsection{The case $0<q<p= \infty$}
Let \begin{eqnarray}\label{hfunction}
h(x):=\sum_{n=10}^{\infty}{\eta(x-2^{\zeta_n}e_1)e^{2\pi i\langle x,2^{\zeta_n}e_1\rangle}}.
\end{eqnarray}
Note that $h$ is defined by letting $b_{\zeta_n}=1$ in (\ref{gfunction}). Therefore
one has  for $\zeta_m-1\leq k\leq \zeta_m+1, m\geq 10$ \begin{eqnarray}\label{esth}
|\Pi_kh(x)|\lesssim \frac{1}{\big( 1+|x-2^{\zeta_m}e_1|\big)^M}
\end{eqnarray}
and $\Vert h\Vert_{\dot{F}_{\infty}^{0,q}}\lesssim 1$.
Moreover, \begin{eqnarray*}
T_mh(x)=\eta(x)\sum_{n=10}^{\infty}{e^{2\pi i\langle x,2^{\zeta_m}e_1\rangle}}
\end{eqnarray*} and \begin{eqnarray*}
\Vert T_mh\Vert_{\dot{F}_{\infty}^{0,q}}=\infty.
\end{eqnarray*}

\subsection{{Proof of Theorem \ref{sharpthm1}}}
Let $r=\min{(p,q)}$ and suppose that $r<1$. 
\subsubsection{Construction of a multiplier}
For $u,s,\tau>0$ let  $$h^{(u,s,\tau)}(x):=\dfrac{1}{(1+|x|^2)^{1/2(s+d/u)}}\dfrac{1}{\big( 1+\log{(1+|x|)}\big)^{\tau}}.$$
Then we define 
\begin{eqnarray*}
K^{(u,s,\tau)}(x):=h^{(u,s,\tau)}\ast \eta(x)e^{-2\pi i \langle x,e_1\rangle}
\end{eqnarray*}
 and 
\begin{eqnarray*}
m^{(u,s,\tau)}(\xi):=\sum_{n=10}^{\infty}{\widehat{K^{(u,s,\tau)}_n}(\xi)}
\end{eqnarray*} where $K^{(u,s,\tau)}_n(x):=2^{\zeta_nd}K^{(u,s,\tau)}(2^{\zeta_n}x)$.
Then it should be observed that 
\begin{eqnarray*}
m^{(u,s,\tau)}(2^k\xi)\varphi(\xi)=\begin{cases}
\widehat{K^{(u,s,\tau)}}(2^{k-\zeta_n}\xi)\varphi(\xi),&\zeta_n-2\leq k\leq \zeta_n+2,n\geq 10\\
0, & otherwise.
\end{cases}
\end{eqnarray*}
Let $k=\zeta_n-j$ for some $j\in\{-2,-1,0,1,2\}$. Then
\begin{eqnarray*}
\big|\big(m^{(u,s,\tau)}(2^k\cdot)\varphi \big)^{\vee}(x) \big|&\lesssim&\big(|h^{(u,s,\tau)}|\ast|\eta|(2^j\cdot) \big)\ast|\varphi^{\vee}|(x)\lesssim h^{(u,s,\tau)}(x)
\end{eqnarray*}
where one has to use the fact that $h^{(u,s,\tau)}(x+y)\leq \dfrac{h^{(u,s,\tau)}(x)}{h^{(u,s,\tau)}(y)}$.
Therefore  one has 
\begin{eqnarray*}
\big\Vert \big(m^{(u,s,\tau)}(2^k\cdot)\varphi \big)^{\vee}\big\Vert_{L^u(B_l)}&\lesssim&\Vert h^{(u,s,\tau)}\Vert_{L^u(B_l)}\lesssim 2^{-ls}(1+l)^{-\tau}
\end{eqnarray*} uniformly in $k$.
This leads to
\begin{eqnarray*}
\big\Vert \big(m^{(u,s,\tau)}(2^k\cdot)\varphi\big)^{\vee}\big\Vert_{K_u^{s,t}}
&\lesssim& \Big( \sum_{l=0}^{\infty}{(1+l)^{-\tau t}}\Big)^{1/t}.
\end{eqnarray*} 

Finally, if $\tau>1/t$ then $\mathcal{K}_{u}^{s,t}[m^{(u,s,\tau)}]<\infty$.

\subsubsection{ The case $0< p<1$ and $p\leq q$ }
Let $f(x):=2^{\zeta_{10}d}\widetilde{\eta}(2^{\zeta_{10}}x)e^{-2\pi i\langle 2^{\zeta_{10}}x,e_1\rangle}$. 
It clearly follows that  $\Vert f\Vert_{\dot{F}_{p}^{0,q}}\lesssim 1$. 

On the other hand
\begin{eqnarray*}
T_{m^{(u,s,\tau)}}f(x)=2^{\zeta_{10}d}h^{(u,s,\tau)}\ast\eta(2^{\zeta_{10}}x)e^{-2\pi i\langle 2^{\zeta_{10}}x,e_1 \rangle}
\end{eqnarray*}
and 
\begin{eqnarray*}
\Vert T_{m^{(u,s,\tau)}}f\Vert_{\dot{F}_p^{0,q}}&\gtrsim& \Big\Vert \Big(\sum_{k\in\mathbb{Z}}{\big|\big(\phi_{k-1}+\phi_k+\phi_{k+1} \big)\ast T_{m^{(u,s,\tau)}}f \big|^q} \Big)^{1/q}\Big\Vert_{L^p}\\
&\geq&\Vert h^{(u,s,\tau)}\ast \eta\Vert_{L^p}\gtrsim \Vert h^{(u,s,\tau)}\Vert_{L^p}.
\end{eqnarray*}
This diverges if $s<d/p-d/u$ or if $s=d/p-d/u$ and $\tau\leq 1/p$.
Hence if $\tau=d/p-d/u$ and $t>p$, then choose $\tau$ so that $1/t<\tau<1/p$.

\subsubsection{The case $0< q<1$ and $q\leq p<\infty$ }
In this case we apply the idea in \cite{Ch_Se}.
 Let $\{r_k\}_{k=1}^{\infty}$ be a fixed sequence of positive numbers. Suppose $0<q\leq p<\infty$ and $K\in\mathcal{E}(1)$. If
\begin{eqnarray*}
\Big\Vert \Big( \sum_{k}{\big| r_k^dK(r_k\cdot)\ast f_k\big|^q}\Big)^{1/q}\Big\Vert_{L^p}\leq A\Big\Vert \Big(  \sum_{k}{|f_k|^q}\Big)^{1/q}\Big\Vert_{L^p}
\end{eqnarray*} for all $f_k\in\mathcal{E}(r_k)$ with $\{f_k\}\in L^p(l^q)$, then
there exists a constant $C_{p,q,d}$ such that
\begin{eqnarray*}
\Vert K\Vert_{L^q}\leq C_{p,q,d}A.
\end{eqnarray*}

By applying Minkowski's inequality with $1/p>1$ one obtains
\begin{eqnarray*}
\Vert K^{(u,s,\tau)}\Vert_{L^q}\gtrsim \Vert h^{(u,s,\tau)}\Vert_{L^q}
\end{eqnarray*}  and this diverges if $s<d/q-d/u$ or if $s=d/q-d/u$ and  $\tau\leq 1/q$. Therefore the proof is done by choosing $\tau$ satisfying $1/t<\tau<1/q$ if $s=d/q-d/u$ and $t>q$.

\subsection{{Proof of Theorem \ref{sharpthm2}}}\label{sharpnesstheorem1.2}
Now suppose $1\leq p,q\leq \infty$, $1<u\leq \infty$, and $s\in\mathbb{R}$.
\subsubsection{Construction of  multipliers}

Let
\begin{eqnarray*}
m^{(1)}(\xi)=\sum_{n=10}^{\infty}{{2^{-\zeta_n(s-d(1-1/u))}}{{\zeta_n}^{-s/d}}\widehat{\widetilde{\eta}}(\xi-2^{\zeta_n}e_1)e^{2\pi i\langle \lambda_{\zeta_n},\xi \rangle}}
\end{eqnarray*}
\begin{eqnarray*}
m_{\alpha}^{(2)}(\xi)=\sum_{n=10}^{\infty}{2^{-2\zeta_n(s-d(1-1/u))}{\zeta_n}^{-s\alpha }\widehat{\widetilde{\eta}}\big(2^{\zeta_n}(\xi-2^{\zeta_n}e_1)\big)e^{2\pi i\langle 2^{\zeta_n} {\zeta_n}^{\alpha}e_1,\xi\rangle}}.
\end{eqnarray*}

We claim that 
\begin{eqnarray}\label{u11}
\mathcal{K}_{u}^{s,t}[m^{(1)}]<\infty, \quad s,t>0,~ u>1
\end{eqnarray} 
\begin{eqnarray}\label{u2}
\mathcal{K}_{u}^{s,t}[m_{\alpha}^{(2)}]<\infty, \quad \alpha>0.
\end{eqnarray}

Indeed, one has
\begin{eqnarray*}
&&m^{(1)}(2^k\xi)\varphi(\xi)\\
&=&\begin{cases}
2^{-\zeta_n(s-d+d/u)}{\zeta_n^{-s/d}}\widehat{\widetilde{\eta}}(2^k\xi-2^{\zeta_n}e_1)e^{2\pi i\langle 2^k\lambda_{\zeta_n},\xi\rangle}\varphi(\xi), &\zeta_n-2\leq k\leq \zeta_n+2, n\geq 10\\
0, & otherwise.
\end{cases}
\end{eqnarray*}
Let $k=\zeta_n-j$ for $j\in\{-2,-1,0,1,2\}$. Then one obtains
\begin{eqnarray*}
\big|\big(m^{(1)}(2^k\cdot)\varphi \big)^{\vee}(x) \big|&\leq& 2^{jd}2^{-\zeta_n(s+d/u)}{\zeta_n^{-s/d}}\big|\widetilde{\eta}(\cdot/2^{\zeta_n-j}+\lambda_{\zeta_n}) \big|\ast |\varphi^{\vee}|(x)\\
&\lesssim_M&2^{-\zeta_n(s+d/u)}\zeta_n^{-s/d}\dfrac{1}{\big(1+|\frac{x}{2^{\zeta_n}}+\lambda_{\zeta_n}| \big)^{2M}}
\end{eqnarray*}
and for sufficiently large $C>0$
\begin{eqnarray*}
\big\Vert \big(m^{(1)}(2^k\cdot)\varphi \big)^{\vee}\big\Vert_{L^u(B_l)}\lesssim \begin{cases}
2^{-\zeta_ns}\zeta_n^{s/d},    &\\
2^{-\zeta_ns}\zeta_n^{s/d}2^{\zeta_nM}2^{-lM},  &2^l>C2^{\zeta_n}\zeta_n^{1/d}.
\end{cases}
\end{eqnarray*}
This yields (\ref{u11}).

In order to prove (\ref{u2})
it should be observed that
\begin{eqnarray*}
&&m_{\alpha}^{(2)}(2^k\xi)\varphi(\xi)\\
&=&\begin{cases}
2^{-2\zeta_n(s-d+d/u)}\zeta_n^{-s\alpha}\widehat{\widetilde{\eta}}\big(2^{\zeta_n}(2^k\xi-2^{\zeta_n}e_1) \big)e^{2\pi i\langle 2^{\zeta_n+k}\zeta_n^{\alpha}e_1,\xi\rangle}\varphi(\xi), & \zeta_n-2\leq k\leq \zeta_n+2, n\geq 10\\
0, & otherwise.
\end{cases}
\end{eqnarray*}
Let $k=\zeta_n-j$ for some $j\in \{-2,-1,0,1,2\}$. Then
one has
\begin{eqnarray*}
\big| \big( m_{\alpha}^{(2)}(2^k\cdot)\varphi\big)^{\vee}\big|&\lesssim& 2^{-2\zeta_n(s+d/s)}\zeta_n^{-s\alpha}\big| \widetilde{\eta}\big(\cdot/2^{2\zeta_n-j}+\zeta_n^{\alpha}e_1 \big)\big|\ast |\varphi^{\vee}|(x)\\
&\lesssim_M&2^{-2\zeta_n(s+d/u)}\zeta_n^{-s\alpha}\dfrac{1}{\big(1+\big|\frac{x}{2^{2\zeta_n}}+\zeta_n^{\alpha}e_1 \big| \big)^{2M}}
\end{eqnarray*}
and for sufficiently large $C>0$
\begin{eqnarray*}
\big\Vert \big(m_{\alpha}^{(2)}(2^k\cdot)\varphi \big)^{\vee}\big\Vert_{L^u(B_l)}\lesssim \begin{cases}
2^{-2\zeta_n s}\zeta_n^{-s\alpha}, &\\
2^{-2\zeta_n s}\zeta_n^{-s\alpha}2^{2\zeta_n M}2^{-lM}, & 2^l>C2^{2\zeta_n}\zeta_n^{\alpha}.
\end{cases}
\end{eqnarray*}
This proves (\ref{u2}).

\subsubsection{The case $1\leq p<q\leq \infty$}
Let $1/q<\epsilon<1/p$ and  $f\in \dot{F}_p^{0,q}$ be defined as in (\ref{ffunction}).  
Then 
\begin{eqnarray*}
T_{m^{(1)}}f(x)=\sum_{n=10}^{\infty}{a_{\zeta_n}2^{-\zeta_n(s-d+d/u)}\zeta_n^{-s/d}\eta(x+\lambda_{\zeta_n})e^{2\pi i\langle x+\lambda_{\zeta_n},2^{\zeta_n}e_1\rangle}}
\end{eqnarray*}
and
\begin{eqnarray*}
\Vert T_{m^{(1)}}f\Vert_{\dot{F}_p^{0,q}}&\gtrsim& \Big\Vert \Big(\sum_{k\in\mathbb{Z}}{\big| \big(\phi_{k-1}+\phi_k+\phi_{k+1} \big)\ast T_{m^{(1)}}f\big|^q} \Big)^{1/q}\Big\Vert_{L^p}\\
&=&\Big(\int_{\mathbb{R}^d}{\Big(\sum_{n=10}^{\infty}{|a_{\zeta_n}|^q2^{-\zeta_nq(s-d+d/u)}\zeta_n^{-sq/d}\big|\eta(x+\lambda_{\zeta_n}) \big|^q} \Big)^{p/q}}dx \Big)^{1/p}\\
&\geq&\Big(\sum_{n=10}^{\infty}{\int_{|x+\lambda_{\zeta_n}|\leq 1/100}{|a_{\zeta_n}|^p2^{-\zeta_np(s-d+d/u)}\zeta_n^{-sp/d}\big|\eta(x+\lambda_{\zeta_n})\big|^p}dx} \Big)\\
&\gtrsim&\Big( \sum_{n=10}^{\infty}{2^{-\zeta_np(s-d+d/u)}\zeta_n^{-p(s/d+1/q)}(\log{\zeta_n})^{-\epsilon p}}\Big)^{1/p}.
\end{eqnarray*}
This diverges if $s<d(1-1/u)$,  or if $s=d(1-1/u)$ and $1/p-1/q\geq 1-1/u$.
One has proved the case $ u\leq q\leq \infty$ of (2) and the case $1\leq p<q\leq \infty$ of (1) and (4).

\subsubsection{The case $1\leq q<p<\infty$}\label{q<1}
Choose $b_k:=k^{-1/p}(\log{k})^{-\delta}$ for $1/p<\delta <1/q$ and let
\begin{eqnarray*}
g(x):=\sum_{n=10}^{\infty}{b_{\zeta_n}\eta(x-\lambda_{\zeta_n})e^{2\pi i\langle x,2^{\zeta_n}e_1\rangle}}.
\end{eqnarray*}
Then by the same argument in (\ref{gfunction}) one can show $g\in\dot{F}_p^{0,q}$.
Moreover one has
\begin{eqnarray*}
T_{m^{(1)}}g(x)=\eta(x)\sum_{n=10}^{\infty}{b_{\zeta_n}2^{-\zeta_n(s-d+d/u)}\zeta_n^{-s/d}e^{2\pi i\langle x+\lambda_{\zeta_n},2^{\zeta_n}e_1\rangle}}
\end{eqnarray*}
 and 
 \begin{eqnarray*}
 \Vert T_{m^{(1)}}g\Vert_{\dot{F}_p^{0,q}}&\gtrsim& \Big\Vert \Big(\sum_{k\in\mathbb{Z}}{\big| \big(\phi_{k-1}+\phi_k+\phi_{k+1} \big)\ast T_{m^{(1)}}g\big|^q} \Big)^{1/q}\Big\Vert_{L^p}\\
 &\geq&\Vert \eta\Vert_{L^p}\Big( \sum_{n=10}^{\infty}{|b_{\zeta_n}|^q2^{-\zeta_nq(s-d+d/u)}\zeta_n^{-sq/d}}\Big)^{1/q}\\
 &\gtrsim&\Big( \sum_{n=10}^{\infty}{2^{-\zeta_nq(s-d+d/u)}\zeta_n^{-q(1/p+s/d)}\big(\log{\zeta_n}\big)^{-\delta q}}\Big)^{1/q}.
 \end{eqnarray*}
This diverges if $s<d(1-1/u)$, or if $s=d(1-1/u)$ and $1/q-1/p\geq 1-1/u$.
This proves the cases $1\leq q<p<\infty$ of (1) and (4).

\subsubsection{The case $1\leq q<p=\infty$}
For $s>0$ and $\alpha=\frac{1}{sq}$ let
\begin{eqnarray*}
h_{\alpha}(x):=\sum_{n=10}^{\infty}{\eta(x/2^{\zeta_n}-\zeta_n^{\alpha}e_1)e^{2\pi i\langle x,2^{\zeta_n}e_1\rangle}}.
\end{eqnarray*}
Then by using the same arguments in (\ref{esth}) one obtains for $\zeta_n-1\leq k\leq \zeta_n+1, n\geq 10$
\begin{eqnarray*}
|\Pi_kh_{\alpha}(x)|\lesssim \dfrac{1}{\big(1+\big|{x}/{2^{\zeta_n}}-\zeta_n^{\alpha}e_1\big| \big)^{M}}.
\end{eqnarray*}
Therefore one has for sufficiently large $M>0$
\begin{eqnarray*}
\Vert h_{\alpha}\Vert_{\dot{F}_{\infty}^{0,q}}\lesssim \Big\Vert \Big(\sum_{n=10}^{\infty}\sum_{k=\zeta_n-1}^{\zeta_n+1}{\dfrac{1}{\big( 1+| \cdot/2^{\zeta_n}-\zeta_n^{\alpha}e_1|\big)^{Mq}}}\Big)^{1/q}\Big\Vert_{L^{\infty}}
\end{eqnarray*} 
and this is finite because
 for $\alpha>0$
\begin{eqnarray*}
\sum_{k=1}^{\infty}\dfrac{1}{(1+|\cdot/2^k-k^{\alpha}|)^M}
\end{eqnarray*} belongs to $L^{\infty}(\mathbb{R})$ if $M$ is large enough depending on $\alpha$.

On the other hand one has
\begin{eqnarray*}
T_{m_{\alpha}^{(2)}}h_{\alpha}(x)=\sum_{n=10}^{\infty}{2^{-2\zeta_n(s-d(1-1/u))}{\zeta_n}^{-s\alpha}\eta(x/2^{\zeta_n})e^{2\pi i\langle x+2^{\zeta_n}\zeta_n^{\alpha}e_1,2^{\zeta_n}e_1\rangle}}
\end{eqnarray*}
and
\begin{eqnarray*}
\big\Vert T_{m_{\alpha}^{(2)}}h_{\alpha}\big\Vert_{\dot{F}_{\infty}^{0,q}}&\gtrsim&\Big( \int_{[0,1]^d}{\sum_{k=1}^{\infty}{\big|\big( \phi_{k-1}+\phi_k+\phi_{k+1}\big)\ast\big(T_{m_{\alpha}^{(2)}}h_{\alpha}\big)(x) \big|^q}}dx\Big)^{1/q}\\
&\gtrsim&\Big(\int_{[0,1]^d}{\sum_{n=10}^{\infty}{2^{-2q\zeta_n(s-d+d/u)}\zeta_n^{-1}\big| \eta(x/2^{\zeta_n})\big|^q}}dx \Big)^{1/q}.
\end{eqnarray*}
Note that if $\log_2{100\sqrt{d}}\leq \zeta_n$ then $\eta(x/2^{\zeta_n})\geq c$ for $x\in[0,1]^d$. Thus we choose an integer $n_d\geq \frac{\log_2{100\sqrt{d}}}{10}$ and then
\begin{eqnarray*}
\big\Vert T_{m_{\alpha}^{(2)}}h_{\alpha}\big\Vert_{\dot{F}_{\infty}^{0,q}}\gtrsim \Big( \sum_{n=\max{(10,n_d)}}^{\infty}{2^{-2q\zeta_n(s-d+d/u)}\zeta_n^{-1}}\Big)^{1/q}.
\end{eqnarray*}
This diverges if $s\leq d(1-1/u)$.

This proves (3) and the case $1\leq q<p=\infty$ of (1). Furthermore, the case $1<q<\infty$ of (2) is proved via duality (\ref{dual}).

\subsubsection{The proof of Theorem \ref{sharpthm2} (5) }
This can be proved by using an idea in the proof of Theorem \ref{sharpbesov} (2).

\section{\textbf{Proof of Theorem \ref{weightmultiplier} and \ref{sharptheorem1} }}\label{weightproof}
Suppose $0<p,q\leq \infty$, $p\not= q$, and $0<u\leq\min{(1,p,q)}$.

\subsection{Proof of Theorem \ref{weightmultiplier}}
\subsubsection{The case $0<q<p<\infty$}
Let $0<q<p<\infty$, and $0<u\leq q$.
By (\ref{dgdgdg}) it suffices to prove
\begin{eqnarray*}
\big\Vert \sup_{P:x\in P\in\mathcal{D}}{\mathcal{R}_{P,u}^{out}}\big\Vert_{L^p(x)}\lesssim \mathcal{B}_{p,q}(w)\mathcal{K}_u^{0,u}(w)[m].
\end{eqnarray*}
Let $P$ be a dyadic cube with side length $l(P)=2^{-\mu}$.
Using H\"older's inequality with $p/q>1$ and (\ref{tftf}), $\mathcal{R}_{P,u}^{out}$ is controlled by
\begin{eqnarray*}
&& \mathcal{B}_{p,q}(w)\Big( \dfrac{1}{|P|}\int_P{\Big( \sum_{k=\mu}^{\infty}{w(k-\mu)^p2^{kdp(1/u-1)}\big\Vert m_k^{\vee}\Pi_kf(y-\cdot)\big\Vert_{L^u(B(0,2l(P))^c)}^p}\Big)^{q/p}}dy\Big)^{1/q}\\
&\lesssim&\mathcal{B}_{p,q}(w)\Big(\dfrac{1}{|P|}\int_P{\Big(\sum_{k=\mu}^{\infty}{2^{kdp(1/u-1)}\Big(\sum_{l=k-\mu}^{\infty}{w(l)^u\big| m_k^{\vee}\chi_{B_l}(2^k\cdot)\big|^u\ast\big| \Pi_kf\big|^u(y)} \Big)^{p/u}} \Big)^{q/p}}dy \Big)^{1/q}.
\end{eqnarray*}
Then by taking supremum over $P\in\mathcal{D}$ containing $x$ one obtains
\begin{eqnarray*}
&&\big\Vert \sup_{P:x\in P\in\mathcal{D}}{\mathcal{R}_{P,u}^{out}}\big\Vert_{L^p(x)}\\
&\lesssim&\mathcal{B}_{p,q}(w)\Big\Vert  \mathcal{M}_{q/p}\Big(\sum_{k\in\mathbb{Z}}{2^{kdp(1/u-1)}\Big(\sum_{l=0}^{\infty}{w(l)^u\big| m_k^{\vee}\chi_{B_l}(2^k\cdot)\big|^u\ast\big| \Pi_kf\big|^u} \Big)^{p/u}} \Big)\Big\Vert_{L^1}^{1/p}\\
&\lesssim&\mathcal{B}_{p,q}(w)\Big(\sum_{k\in\mathbb{Z}}{2^{kdp(1/u-1)}\Big\Vert \sum_{l=0}^{\infty}{w(l)^u\big| m_k^{\vee}\chi_{B_l}(2^k\cdot)\big|^u\ast |\Pi_kf|^u}\Big\Vert_{L^{p/u}}^{p/u}} \Big)^{1/p}\\
&\lesssim&\mathcal{B}_{p,q}(w)\Big(\sum_{k\in\mathbb{Z}}{\Vert \Pi_kf\Vert_{L^p}^p\Big( \sum_{l=0}^{\infty}{w(l)^u\big\Vert \big(m(2^k\cdot)\varphi \big)^{\vee}\big\Vert_{L^u(B_l)}^u}\Big)^{p/u}} \Big)^{1/p}\\
&\lesssim&\mathcal{B}_{p,q}(w)\mathcal{K}_{u}^{0,u}(w)[m]\Vert f\Vert_{\dot{F}_p^{0,p}}
\end{eqnarray*}
where the second inequality follows from the boundedness of $\mathcal{M}_{q/p}$ in $L^1$ and the thrid one is from  Minkowski's inequality with $p/u>1$ and Young's inequality.
Then embedding $\dot{F}_p^{0,q}\hookrightarrow \dot{F}_p^{0,p}$ finishes the proof.

\subsubsection{The case $0<q<p=\infty$}
Due to (\ref{xbxbxb}) one needs to prove that
\begin{eqnarray}\label{etetet}
\sup_{P\in\mathcal{D}}{\mathcal{R}_{P,u}^{out}}\lesssim \mathcal{B}_{\infty,q}(w)\mathcal{K}_{u}^{0,u}(w)[m]\Vert f\Vert_{\dot{F}_{\infty}^{0,q}}.
\end{eqnarray}
By (\ref{tftf}) and H\"older's inequality(if $u<q$)  $\mathcal{R}_{P,u}^{out}$ is controlled by
\begin{eqnarray*}
&&\Big( \dfrac{1}{|P|}\int_P{\sum_{k=\mu}^{\infty}{w(k-\mu)^{-q}2^{kdq(1/u-1)}\Big(\sum_{l=k-\mu}^{\infty}{w(l)^u\big| m_k^{\vee}\chi_{B_l}(2^k\cdot)\big|^u\ast |\Pi_kf|^u(x)} \Big)^{q/u}}}dx\Big)^{1/q}\\
&\lesssim&\mathcal{B}_{\infty,q}(w)\sup_{k\in\mathbb{Z}}{2^{kd(1/u-1)}\Big( \sum_{l=k-\mu}^{\infty}{w(l)^u\big\Vert \big| m_k^{\vee}\chi_{B_l}(2^k\cdot)\big|^u\ast |\Pi_kf|^u\big\Vert_{L^{\infty}}}\Big)^{1/u}}.
\end{eqnarray*}
Then it follows that  $\mathcal{R}_{P,u}^{out}\lesssim \mathcal{B}_{\infty,q}(w)\mathcal{K}_{u}^{0,u}(w)[m]\Vert f\Vert_{\dot{F}_{\infty}^{0,\infty}}$ by Young's inequality. Finally,
 (\ref{etetet}) is proved from embedding $\dot{F}_{\infty}^{0,q}\hookrightarrow \dot{F}_{\infty}^{0,\infty}$.

\subsubsection{The case $0<p<q\leq \infty$}
The case $1<p<q\leq \infty$ follows via duality (\ref{dual}).
Therefore we suppose $0<p\leq 1$.
We follow the proof of Theorem \ref{multipliertheorem0}.
Let $A_{Q_0,k}$ be defined as (\ref{rq}). Then one needs to show 
\begin{eqnarray}\label{wegoal}
\Big(\sum_{j=2}^{\infty}\int_{D_j}{\Big(\sum_{k=\mu}^{\infty}{2^{kdq(1/u-1)}\big(\mathcal{L}_{u,j}^{mid}[A_{Q_0,k}](x) \big)^q} \Big)^{p/q}}dx \Big)^{1/p}\lesssim \mathcal{B}_{p,q}(w)\mathcal{K}_{u}^{0,u}(w)[m]
\end{eqnarray}
because other terms are bounded by $\mathcal{K}_{u}^{0,u}[m]$.

By $l^p\hookrightarrow l^q$, Minkowski's inequality (with $p/u>1$ if $p>u$),  $l^u\hookrightarrow l^p$, the increasing property of $\{w(l)\}$, and H\"older's inequality with $q/p>1$,
the left hand side of (\ref{wegoal}) is less than
\begin{eqnarray*}
&&\Big(\sum_{k=\mu}^{\infty}\sum_{j=2}^{\infty}{2^{kdp(1/u-1)}\int_{D_j}{\Big(\int_{\mathbb{R}^d}{\sum_{l:2^l\approx  2^{j+k-\mu}}{\big| \big(m_k^{\vee}\chi_{B_l}(2^k\cdot) \big)(y)\big|^u} \big|A_{Q_0,k}(x-y) \big|^u}dy \Big)^{p/u}}dx} \Big)^{1/p}\\
&\lesssim&\Big( \sum_{k=\mu}^{\infty}{2^{kdp(1/u-1)}\Vert A_{Q_0,k}\Vert_{L^p}^p\sum_{j=2}^{\infty}\Big( \sum_{l:2^l\approx 2^{j+k-\mu}}{\big\Vert m_k^{\vee}\chi_{B_l}(2^k\cdot)\big\Vert_{L^u}^u}\Big)^{p/u}}\Big)^{1/p}\\
&\lesssim&\Big( \sum_{k=\mu}^{\infty}{2^{kdp(1/u-1)}\Vert A_{Q_0,k}\Vert_{L^p}^p\Big( \sum_{l={k-\mu}}^{\infty}{\big\Vert m_k^{\vee}\chi_{B_l}(2^k\cdot)\big\Vert_{L^u}^u}\Big)^{p/u}}\Big)^{1/p}\\
&\leq&\Big( \sum_{k=\mu}^{\infty}{\Vert A_{Q_0,k}\Vert_{L^p}^pw(k-\mu)^{-p}\Big( \sum_{l={k-\mu}}^{\infty}{\big\Vert \big(m(2^k\cdot)\varphi \big)^{\vee}\big\Vert_{L^u(B_l)}^uw(l)^u}\Big)^{p/u}}\Big)^{1/p}\\
&\lesssim&\mathcal{K}_u^{0,u}(w)[m]\mathcal{B}_{p,q}(w)\Big(\sum_{k=\mu}^{\infty}{\Vert A_{Q_0,k}\Vert_{L^p}^q} \Big)^{1/q}.
\end{eqnarray*}
Then using the fact that $$\big(\sum_{k=\mu}^{\infty}{\Vert A_{Q_0,k}\Vert_{L^p}^q} \big)^{1/q}\lesssim 1$$ (\ref{wegoal}) is proved.

\subsection{Proof of Theorem \ref{sharptheorem1} }
Suppose that $0< p,q,\leq\infty$, $p\not= q$, and $t>0$. We construct a counter example based on the idea in the proof of Theorem \ref{sharpthm0}.
Let \begin{eqnarray*}
w(k):={(1+k)^{|1/p-1/q|}}.
\end{eqnarray*} 
Then it is clear that
$\sum_{k=0}^{\infty}{w(k)^{-s}}<\infty$ for all $s>\frac{1}{|1/p-1/q|}$.

\begin{eqnarray*}
m(\xi):=\sum_{n=10}^{\infty}{\frac{1}{w(\zeta_n)}\widehat{\eta}\big((\xi-2^{\zeta_n}e_1)/2^{\zeta_n}\big)e^{2\pi i\langle 2^{\zeta_n}e_1,\xi-2^{\zeta_n}e_1\rangle}}.
\end{eqnarray*}
Note that $m$ is a modification of (\ref{fffexample}) with weight $\{\frac{1}{w}\}$. 
Thus, by using the idea in (\ref{estfff}) one obtains that for $k=\zeta_m-j$, $j\in \{-2,-1,0,1,2\}$
\begin{eqnarray*}
\big\Vert \big(m(2^k\cdot)\varphi \big)^{\vee}\big\Vert_{L^u(B_l)}\lesssim_N\begin{cases}
w(\zeta_m)^{-1}&\\
2^{-lN}w(\zeta_m)^{-1}& 2^l\geq 2^{2\zeta_m+3}\\
2^{-\zeta_m N}w(\zeta_m)^{-1}& 2^l\leq 2^{2\zeta_m-4}
\end{cases}
\end{eqnarray*}
and \begin{eqnarray*}
&&\big\Vert \big( m(2^k\cdot)\varphi\big)^{\vee}\big\Vert_{K_u^{0,t}(w)}\\
&\lesssim&\Big( \sum_{l=0}^{2\zeta_m-4}{\Big(\frac{w(l)}{w(\zeta_m)} \Big)^t2^{-\zeta_mNt}}\Big)^{1/t}+\Big( \sum_{l=0}^{2\zeta_m-4}{\Big(\frac{w(l)}{w(\zeta_m)} \Big)^t2^{-\zeta_mNt}}\Big)^{1/t}+\Big( \sum_{l=0}^{2\zeta_m-4}{\Big(\frac{w(l)}{w(\zeta_m)} \Big)^t2^{-\zeta_mNt}}\Big)^{1/t}\\
&\lesssim&\frac{w(2\zeta_m+3)}{w(\zeta_m)}\lesssim 1
\end{eqnarray*}
by the monotonicity of $\{w(k)\}$. This proves $\mathcal{K}_u^{0,t}(w)[m]<\infty$.

Now let us prove that the operator $T_m$ is not bounded on $F_p^{0,q}$ for $p\not= q$.
When $0<p<q\leq \infty$ let $f\in \dot{F}_{p}^{0,q}$ be defined as in (\ref{ffunction}) and the similar process yields that
\begin{eqnarray*}
\big\Vert T_mf\big\Vert_{\dot{F}_p^{0,q}}&\gtrsim&\Big( \sum_{n=10}^{\infty}{\Big(\dfrac{|a_{\zeta_n}|}{w(\zeta_n)}\Big)^{p}}\Big)^{1/p}=\Big(\sum_{n=10}^{\infty}{\dfrac{1}{\zeta_n(\log{\zeta_n})^{p\epsilon}}} \Big)^{1/p}=\infty.
\end{eqnarray*}
When $0< q<p<\infty$ let $g\in \dot{F}_p^{0,q}$ be as in (\ref{gfunction}). Then one can similarly prove that
\begin{eqnarray*}
\big\Vert T_mg\big\Vert_{\dot{F}_p^{0,q}}&\gtrsim&\Big( \sum_{n=10}^{\infty}{\Big(\dfrac{|b_{\zeta_n}|}{w(\zeta_n)}\Big)^{q}}\Big)^{1/q}=\Big(\sum_{n=10}^{\infty}{\dfrac{1}{\zeta_n(\log{\zeta_n})^{q\delta}}} \Big)^{1/q}=\infty.
\end{eqnarray*}
When $0<q<p=\infty$ define $h\in \dot{F}_p^{0,q}$ as in (\ref{hfunction}) and it also follows that
\begin{eqnarray*}
\Vert T_mh\Vert_{\dot{F}_{\infty}^{0,q}}&\gtrsim&\Big( \sum_{n=10}^{\infty}{\dfrac{1}{w(\zeta_n)^q}}\Big)^{1/q}=\Big( \sum_{n=10}^{\infty}{\dfrac{1}{\zeta_n}}\Big)^{1/q}=\infty.
\end{eqnarray*}

\section*{Acknowledgement}
 {The author would like to thank his graduate advisor Andreas Seeger for the guidance and helpful discussions. The author was supported in part by NSF grant DMS 1500162.}

\end{document}